\documentclass[10 pt]{amsart}

\usepackage{amscd,amsthm,amsfonts,amssymb,amsmath}
\usepackage[OT2,T1]{fontenc}
\usepackage{pst-node}
\usepackage{tikz}
\usetikzlibrary{matrix,arrows}
\usepackage[all]{xy}
\usepackage[margin=1.5in]{geometry}

\usepackage{textcomp}
\usepackage{array}
\usepackage{mathrsfs}
\usepackage{enumerate}

    \theoremstyle{plain}
    \newtheorem{thm}{Theorem}[section] \newtheorem{cor}[thm]{Corollary}
    \newtheorem{lem}[thm]{Lemma}  \newtheorem{prop}[thm]{Proposition}
    \newtheorem {conj}[thm]{Conjecture} \newtheorem{defn}[thm]{Definition}
\newtheorem{fact}[thm]{Fact}
\newtheorem{eg}[thm]{Example}
\newtheorem*{thm*}{Theorem}
\newtheorem*{prop*}{Proposition}
\newtheorem*{lem*}{Lemma}
\newtheorem*{cor*}{Corollary}

\theoremstyle{remark} \newtheorem{remark}[thm]{Remark}
\theoremstyle{remark}

\numberwithin{equation}{section}

\DeclareSymbolFont{cyrletters}{OT2}{wncyr}{m}{n}
\DeclareMathSymbol{\Sha}{\mathalpha}{cyrletters}{"58}

\title[On the $p$-part of the BSD conjecture for certain CM elliptic curves]{On the $p$-part of the Birch--Swinnerton-Dyer conjecture for elliptic curves with complex multiplication by the ring of integers of $\mathbb{Q}(\sqrt{-3})$}
\author{YUKAKO KEZUKA}

\begin{document}

\pagestyle{plain}
\setcounter{page}{1}

\maketitle

\begin{abstract}
We study infinite families of quadratic and cubic twists of the elliptic curve $E = X_0(27)$. For the family of quadratic twists, we establish a lower bound for the $2$-adic valuation of the algebraic part of the value of the complex $L$-series at $s=1$, and, for the family of cubic twists, we establish a lower bound for the $3$-adic valuation of the algebraic part of the same $L$-value. We show that our lower bounds are precisely those predicted by the celebrated conjecture of Birch and Swinnerton-Dyer.
\end{abstract}

\section{Introduction}
Let $E$ be an elliptic curve defined over $\mathbb{Q}$, and let $L(E,s)$ denote its complex $L$-series. We assume that $L(E,1)\neq 0$.  Then, by a well-known theorem of Kolyvagin, both $E(\mathbb{Q})$ and the Tate--Shafarevich group $\Sha(E)$ of $E$ over $\mathbb{Q}$ are finite. 
Define
$$L^{(\text{alg})}\left(E,1\right)=\frac{L\left(E,1\right)}{\Omega c_\infty\O},$$
where $c_\infty$ denotes the number of connected real components of $E(\mathbb{R})$, and $\Omega$ is the least positive real period of the N\'{e}ron differential of any global Weierstrass minimal equation for $E$. It is well-known that $L^{(\text{alg})}\left(E,1\right)$ is a rational number. For a prime $q$ of bad reduction for $E$, define
$$c_q=[E(\mathbb{Q}_q): E^0(\mathbb{Q}_q)],$$
where $E^0(\mathbb{Q}_q)$ denotes the subgroup of $E(\mathbb{Q}_q)$ consisting of all points with non-singular reduction modulo $q$. The Birch--Swinnerton-Dyer conjecture for $E$ asserts that:
\begin{conj}\label{conj1}
\begin{equation}\label{c1}L^{(\text{alg})}\left(E,1\right)=\frac{\#(\Sha(E))\prod\limits_{q\text{ bad}}c_q}{\#(E(\mathbb{Q}))^2}.
\end{equation}
\end{conj}

Since both sides of \eqref{c1} are rational numbers, Conjecture \ref{conj1} clearly implies that:
\begin{conj}\label{conj2} For each prime number $p$, we have
\begin{equation}\label{c2}\mathrm{ord}_p\left(L^{(\text{alg})}\left(E,1\right)\right)=\mathrm{ord}_p\left(\frac{\#\left(\Sha(E)(p)\right)}{\#\left(E(\mathbb{Q})(p)\right)^2}\right)+\mathrm{ord}_p\left(\prod\limits_{q\text{ bad}}c_q\right).
\end{equation}
\end{conj}
When $E$ has complex multiplication, Rubin establishes \eqref{c2} in \cite[Theorem 11.1]{rubin1} for all primes $p$ which do not divide the order $w$ of the group of roots of unity in the field of complex multiplication. However, these methods at present seem very difficult to apply for primes $p$ which divide $w$, except when $E$ has potential ordinary reduction at such a \mbox{prime $p$.} 

In the present paper, we shall consider the quadratic and cubic twists of the curve
\begin{equation}\label{eq0}
E = X_0(27): \; Y^2+Y=X^3-7,
\end{equation}
which has conductor 27 and admits complex multiplication by the full ring of integers $\mathcal{O}_K=\mathbb{Z}[\omega]$, where $\omega=\frac{-1+\sqrt{-3}}{2}$, of  the field $ K = \mathbb{Q}(\sqrt{-3})$. The associated classical Weierstrass equation is
$$E: y^2=4x^3-3^3,$$
which we obtain by the change of variables
\begin{align*}
     x&=X\\
     y&=2Y+1.
\end{align*}
Note that $c_\infty=1$ for $E$, so that  $L^{(\text{alg})}\left(E,1\right)=\frac{L\left(E,1\right)}{\Omega}$. It is easily shown that \mbox{$L^{(\text{alg})}\left(E,1\right)=\frac{1}{3}$.} On the other hand, classical descent theory proves that $E(\mathbb{Q})=\{\mathcal{O}, (3,\pm3^2)\}\cong \mathbb{Z}/3\mathbb{Z}$ and $\Sha(E)(2)=\Sha(E)(3)=0$. Combining this with \cite[Theorem 11.1]{rubin1}, we conclude that Conjecture \ref{conj1} is valid for $E$. 

 Given an integer $\lambda>1$,  let $E(\lambda)$ denote the elliptic curve
$$E(\lambda): y^2=4x^3-3^3\lambda .$$
First, we consider the case when $\lambda=D^3$, for a square-free positive integer $D$, so that $E(D^3)$ is the twist of $E$ by the quadratic extension $\mathbb{Q}(\sqrt{D})/\mathbb{Q}$.
We define a rational prime number $p$ to be a \emph{special split} prime for $E$ if it splits completely in the field $K(x(E[4]))$, the field obtained by adjoining to $K$ the $x$-coordinates of all non-zero points in $E[4]$, the group of $4$-division points on $E$. In fact, we have that $K(x(E[4]))=K(\boldsymbol{\mu}_{4}, \sqrt[3]{2})$. Moreover, the theory of complex multiplication provides the following alternative description of the set of special split primes. Let $\psi$ denote the Gr\"{o}ssencharacter of $E$ over $K$. Then a prime $p$ is special split if and only if it splits in $K$, and  $\psi(\mathfrak{p})\equiv \pm 1\bmod 4$ for both of the primes $\mathfrak{p}$ of $K$ above $p$ (see Lemma \ref{lem4} of Appendix A). 
Our aim in this first section of the paper is to prove:

\begin{thm*} Let $D>1$ be an integer which is a square-free product of special split primes. Then 
$$\mathrm{ord}_2\left(L^{(\text{alg})}\left(E(D^3),1\right)\right)\geqslant 2k(D),$$
where $k(D)$ is the number of prime factors of $D$.
\end{thm*}
This bound is sharp, as we will see in Remark \ref{rem2}. Some numerical examples are listed in Appendix B.
We  show later, using Tate's algorithm, that
$$\mathrm{ord}_2\left(\frac{\prod\limits_{q\text{ bad}}c_q}{\#(E(D^3)(\mathbb{Q}))^2}\right)=2k(D).$$
Hence the $2$-part of the Birch--Swinnerton-Dyer conjecture predicts that if \mbox{$L(E(D^3),1)\neq 0$}, then
$$\mathrm{ord}_2\left(L^{(\text{alg})}(E(D^3),1)\right)=2k(D)+\mathrm{ord}_2\left(\#\Sha(E(D^3))\right).$$
In particular, it predicts that equality occurs in  the lower bound of this theorem if and only if \mbox{$\mathrm{ord}_2\left(\#\Sha\left(E(D^3)\right)\right)=0$}.

Next consider the case when $\lambda=D^2$ for a cube-free positive integer $D$, so that $E(D^2)$ is a cubic twist of $E$. We say a prime number $p$ is \emph{cubic-special} if it splits completely in the field $K(E[27])$, but does not split completely in the strictly larger field $K(E[27])((1-\omega)^{1/9})$, where $\omega$ denotes a non-trivial cube root of unity. We then prove:-

\begin{thm*}Let $D>1$ be an integer which is a cube-free product of cubic-special primes. Then 
$$\mathrm{ord}_3\left(L^{(\text{alg})}\left(E(D^2),1\right)\right)\geqslant k(D)+1,$$
where $k(D)$ is the number of distinct prime factors of $D$.
\end{thm*}
Numerical examples show that this lower bound is sometimes sharp.  In fact, the Birch-Swinnerton-Dyer conjecture predicts that the lower bound of this theorem should hold for all odd cube free positive integers $D$ with $D\equiv 1\bmod 9$. Indeed, using Tate's algorithm, it can be shown (see Section 2) that, for all such $D$, we have
$$\mathrm{ord}_3\left(\frac{\prod\limits_{q\text{ bad}}c_q}{\#(E(D^2)(\mathbb{Q}))^2}\right)=k(D)+1.$$
Hence the $3$-part of the Birch--Swinnerton-Dyer conjecture predicts that if \mbox{$L(E(D^2),1)\neq 0$}, we have 
$$\mathrm{ord}_3\left(L^{(\text{alg})}(E(D^2),1)\right)=k(D)+1+\mathrm{ord}_3\left(\#\Sha(E(D^2))\right).$$
In particular, it predicts that equality is attained in the theorem above if and only if \mbox{$\mathrm{ord}_3\left(\#\Sha\left(E(D^2)\right)\right)=0$}. We will prove these theorems by first expressing the value of the complex $L$-series as a sum of Eisenstein series, and then combining an averaging argument over quadratic or cubic twists with an induction on the number of distinct primes divisors. In the case of quadratic twists, this method is essentially due to Zhao \cite{zhao1, zhao2} who established similar results for the congruent number curves with respect to the prime $p = 2$. In Section 4, we will generalise his ideas in order to apply to the cubic twists of $E$ with respect to the prime $p=3$. All numerical examples in this paper are computed using the computer
package Magma.

\subsection*{Acknowledgement} 
This work is part of my PhD thesis at the University of Cambridge, and I wish to thank my supervisor John Coates for suggesting the research topics, giving me many relevant sources of information and patiently giving me guidance and encouragement. I am also very grateful to Jack Lamplugh for many insightful discussions, and to Kevin Buzzard, Tom Fisher, Yoshitaka Hachimori and Shuai Zhai for many helpful comments.

\section{The $p$-part of the Birch--Swinnerton-Dyer Conjecture.}\label{section2.1}

Let $\lambda>1$ be an integer and define $E(\lambda): y^2=4x^3-3^3\lambda$. Let us assume that $L\left(E(\lambda),1\right)\neq 0$, so that $E(\lambda)(\mathbb{Q})$ and $\Sha\left(E(\lambda)\right)$ are finite.  Let $\omega=\frac{-1+\sqrt{-3}}{2}$, a cube root of unity. In this short section, we will compute the Tamagawa factors $c_q$ for the primes $q$ of bad reduction for $E(\lambda)$, and $\mathrm{ord}_p(E(\lambda)(\mathbb{Q}))$ for $p=2$ or $3$ according as $E(\lambda)$ is a quadratic or cubic twist of $E=X_0(27)$.

First, we consider the case when $\lambda=D^3$, for $D>1$ a square-free integer, so that $E(D^3)$ is a quadratic twist of $E$. The primes of bad reduction for $E(D^3)$ are $3$ and the primes dividing $D$, since the discriminant of $E(D^3)$ is $-27D^6$.

\begin{lem} Let $D>1$ be a square-free product of primes coprime to $6$ which split in $\mathbb{Q}(\omega, \sqrt[3]{2})$. Then $$\mathrm{ord}_2\left(\frac{\prod\limits_{q\text{ bad}}c_q}{\#(E(D^3)(\mathbb{Q}))^2}\right)=2 k(D),$$
where $k(D)$ denotes the number of prime factors of $D$.
\end{lem}

\begin{proof}
We will work with the form $y^2=x^3-2^43^3D^3$ which is isomorphic to $E(D^3)$. With the usual notation for Tate's algorithm, we have $a_1=a_3=a_2=a_4=0$, $a_6=-2^43^3D^3$, $b_4=b_8=0$ and $b_6=-2^63^3D^3$. For a bad prime $q$, we have $q\mid a_1, a_2$, $q^2\mid a_3, a_4$ and $q^3\mid a_6$. Let $P_q$ be the polynomial
$$P_q(T)=T^3+\frac{a_6}{q^3}.$$
Then for $q=3$, we have $P_3'(T)=3T^2\equiv 0\bmod 3$ so $P_3(T)$ has a triple root in $\mathbb{Z}/3\mathbb{Z}$. Therefore, $c_3=3$ and $\mathrm{ord}_2(c_3)=0$.  If  $q$ is a prime factor of $D$, then $(P_q(T),P_q'(T))=(T^3+\frac{a_6}{q^3}, 3T^2)=1$ in $\mathbb{Z}/q\mathbb{Z}[T]$, since $3\nmid D$. So $P_q(T)$ has $3$ distinct roots in $\mathbb{Z}/q\mathbb{Z}$. Hence, $c_q=4$ and $\mathrm{ord}_2(c_q)=2$. 

Also, $E(D^3)[2^\infty](\mathbb{Q})=\{\mathcal{O}\}$ since the equation $4x^3-3^3D^3=0$ clearly has no rational solution. 
\end{proof}

Thus \eqref{c2} indeed predicts
\begin{align*}\mathrm{ord}_2\left(L^{\text{(alg)}}\left(E(D^3),1\right)\right)&=\mathrm{ord}_2\left((\Sha\left(E(D^3)\right)[2^\infty]\right)+2 k(D)\\
&\geqslant 2 k(D).
\end{align*}

Next, we consider the case when $\lambda=D^2$, for $D>1$ a cube-free integer, so that $E(D^2)$ is a cubic twist of $E$. We remark that $E(D^2)$ is isomorphic to the curve $x^3+y^3=D$ which is a cubic twist of the Fermat curve $x^3+y^3=1$. The primes of bad reduction for $E(D^2)$ are again $3$ and the primes dividing $D$, since the discriminant of $E(D^2)$ is $-27D^4$.

\begin{lem} Let $D>1$ be an odd, cube-free integer such that $D\equiv 1\bmod 9$ and $D$ is a product of primes congruent to $1$ modulo $3$. Then $$\mathrm{ord}_3\left(\frac{\prod\limits_{q\text{ bad}}c_q}{\#(E(D^2)(\mathbb{Q}))^2}\right)=k(D)+1,$$
where $k(D)$ is the number of distinct prime factors of $D$.
\end{lem}

\begin{proof}
We will work with the form $y^2=x^3-2^43^3D^2$ which is isomorphic to $E(D^2)$. With the usual notation for Tate's algorithm, we have $a_1=a_3=a_2=a_4=0$, $a_6=-2^43^3D^2$, $b_4=b_8=0$ and $b_6=-2^63^3D^2$. Let $q$ be a prime of bad reduction for $E$. If $q$ is a prime factor of $D$, then we have $q\mid a_1, a_2$, $q^2\mid a_3, a_4$ and $q^3\nmid a_6$ hence the type is IV (see \cite[p. 49]{tate}) and $c_{q}=3$ or $1$. However, the polynomial $T^2+\frac{2^43^3D^2}{q^2}$ has roots in $\mathbb{Z}/q\mathbb{Z}$ since $\left(\frac{-3}{q}\right)=(-1)^{q-1}\left(\frac{q}{3}\right)=1$ and so $-\frac{2^43^3D^2}{q^2}$ is a square $\bmod$ $q$. It follows that $c_q=3$ and $\mathrm{ord}_3(c_q)=1$. Otherwise, $q=3$ and we have $3\mid a_1, a_2$, $3^2\mid a_3, a_4$ and $3^3\nmid a_6$. Let $P_3$ be the polynomial
$$P_3(T)=T^3+\frac{a_6}{3^3}.$$
Then $P_3'(T)=3T^2\equiv 0\bmod 3$ so $P_3(T)$ has a triple root in $\mathbb{Z}/3\mathbb{Z}$. After the change of variables $x=X+3D$ the triple root is $0$, and we have $a_1=a_3=0$, $a_2=3^2D$, $a_4=3^3D^2$, $a_6=3^3D^2(D-2^4)\equiv 3 \bmod 9$. So $Y^2-\frac{a_6}{3^4}=Y^2-\frac{D^2(D-2^4)}{3}\equiv Y^2-1\equiv 0\bmod 3$ has distinct roots in $\mathbb{Z}/3\mathbb{Z}$. Hence the type is IV* (see \cite[p. 51]{tate}) and $c_3=3$, so that $\mathrm{ord}_3(c_3)=1$.

Furthermore, by \cite[Exercise 10.19]{sil}, we have \mbox{$E(D^2)(\mathbb{Q})_{\text{tors}} = \{\mathcal{O}\}$} for $D> 1$.
\end{proof}

Thus \eqref{c2} predicts
\begin{align*}\mathrm{ord}_3\left(L^{\text{(alg)}}\left(E(D^2),1\right)\right)&\geqslant \mathrm{ord}_3\left((\Sha\left(E(D^2)\right)[3^\infty]\right)+ k(D) +1\\
&\geqslant k(D) +1.
\end{align*}

\section{Quadratic Twists.}\label{section2.2}
Let  $ K = \mathbb{Q}(\sqrt{-3})$, and write $\boldsymbol{\mu}_K$ for the group of roots of unity in $K$. We fix once and for all an embedding of $K$ into $\mathbb{C}$. In general, if $\lambda$ is a non-zero element of $\mathcal{O}_K$ which is prime to $\#(\boldsymbol{\mu}_K)=6$, we let $\psi_{\lambda}:=\psi_{E(\lambda)/K}$ be the Gr\"{o}ssencharacter of $E(\lambda)$ over $K$ with conductor $\mathfrak{f}$, and let $\mathfrak{g}$ denote some integral multiple of $\mathfrak{f}$. Let $S$ be the set of primes of $K$ dividing $\mathfrak{g}$. We consider the (usually) imprimitive Hecke $L$-series
$$L_S(\overline{\psi}_\lambda,s)=\sum\limits_{(\mathfrak{a},\mathfrak{g})=1}\frac{\overline{\psi}_\lambda(\mathfrak{a})}{(\rm{N}\mathfrak{a})^s}$$
of $\overline{\psi}_{\lambda}$ (the complex conjugate of $\psi_{\lambda}$). It can be defined by the Euler product
$$L_S(\overline{\psi}_\lambda,s)=\prod\limits_{(v,\mathfrak{g})=1}\left(1-\frac{\overline{\psi}_\lambda(v)}{(\mathrm{N} v)^s}\right)^{-1},$$
and if we replace $\mathfrak{g}$ by $\mathfrak{f}$ in the definition, we obtain the primitive Hecke $L$-function $L(\overline{\psi}_\lambda,s)$. In particular, we have
$$L(E(\lambda),1)=L(\overline{\psi}_\lambda,1).$$

Recall that for any complex lattice $L$ and $z, s\in\mathbb{C}$, we can define the Kronecker--Eisenstein series
$$H_1(z,s,L):=\sum\limits_{w\in L}\frac{\overline{z}+\overline{w}}{\:\: |z+w|^{2s}},$$
where the sum in taken over all $w\in L$, except $-z$ if $z\in L$. This series converges for $\rm{Re}(s)>\frac{3}{2}$, and it has an analytic continuation to the whole complex $s$-plane \cite[Theorem 1.1]{gol-sch}. The non-holomorphic Eisenstein series $\mathcal{E}_1^*(z,L)$ is defined by 
$$\mathcal{E}_1^*(z,L):=H_1(z,1,L).$$

Let $\Omega_{\lambda}=\frac{\Omega}{\sqrt[6]{\lambda}}\in \mathbb{C}^\times$, where $\sqrt[6]{\lambda}$ denotes the real root and $\Omega$ is the least positive real period of the N\'{e}ron differential of any global Weierstrass minimal equation for $E$. We write $\mathcal{L}_{\lambda}$ for the period lattice of the curve $E(\lambda)$ over $\mathbb{C}$, and write $\mathcal{L}$ for that of $E$.

Since $\mathfrak{g}$ is a multiple of $\mathfrak{f}$, it follows from \cite[Lemma 3]{coa-wil} that $K(E(\lambda)_{\mathfrak{g}})$, the extension of $K$ obtained by adjoining the coordinates of all $\mathfrak{g}$-division points of $E(\lambda)$ to $K$, is isomorphic to $K(\mathfrak{g})$, the ray class field of $K$ modulo $\mathfrak{g}$. We fix, once and for all,  a set $\mathcal{B}$ of integral ideals of $K$ prime to $\mathfrak{g}$ such that
$$\mathrm{Gal}(K(\mathfrak{g})/K)=\{\sigma_{\mathfrak{b}}\; : \; \mathfrak{b}\in \mathcal{B}\},$$
where the Artin symbol $\sigma_\mathfrak{b}=\left(\mathfrak{b}, K(\mathfrak{g})/K\right)$ of $\mathfrak{b}$ runs over $\mathrm{Gal}\left( K(\mathfrak{g})/K\right)$ precisely once as $\mathfrak{b}$ runs over $\mathcal{B}$. Fix a generator $g$ of $\mathfrak{g}$, so that $\mathfrak{g}=g\mathcal{O}_K$. The next result is due to Goldstein and Schappacher \cite[Proposition 5.5]{gol-sch}.

\begin{lem}\label{lem2.2.1} For all non-zero $\lambda\in\mathcal{O}_K$, we have
$$L_S(\overline{\psi}_{\lambda},s)=\frac{|\Omega_\lambda/g|^{2s}}{\overline{\Omega_\lambda/g}}\sum\limits_{\mathfrak{b}\in\mathcal{B}}H_1\left(\psi_{\lambda}(\mathfrak{b})\frac{\Omega_{\lambda}}{g}, s, \mathcal{L}_{\lambda}\right).$$
\end{lem}

\begin{proof} The Artin map gives an isomorphism
$$\left(\mathcal{O}_K/\mathfrak{g}\right)^\times/\widetilde{\boldsymbol{\mu}}_K \xrightarrow{\sim} \mathrm{Gal}\left(K\left(E(\lambda)_{\mathfrak{g}}\right)/K\right)$$
where $\widetilde{\boldsymbol{\mu}}_K$ denotes the image of the group $\boldsymbol{\mu}_K$ under reduction modulo $\mathfrak{g}$. Moreover, it is clear from the choice of $\lambda$ that the map from $\boldsymbol{\mu}_K$ to $\widetilde{\boldsymbol{\mu}}_K$ is an isomorphism. Hence, the principal ideal $\mbox{$(\psi_{\lambda}(\mathfrak{b})+a)$}$ runs over all integral ideals of $K$ prime to $\mathfrak{g}$ precisely once as $\mathfrak{b}$ runs over $\mathcal{B}$ and $a$ runs over $\mathfrak{g}$. It follows that
\begin{align*}L_S(\overline{\psi}_{\lambda},s)=\sum\limits_{\mathfrak{b}\in\mathcal{B}}\sum\limits_{a\in\mathfrak{g}}\frac{\overline{\psi}_{\lambda}((\psi_{\lambda}(\mathfrak{b})+a))}{|\psi_{\lambda}(\mathfrak{b})+a|^{2s}}.
\end{align*}
Note that since $a\in \mathfrak{g}$, we can write 
$$\psi_{\lambda}(\mathfrak{b})+a=(\psi_{\lambda}(\mathfrak{b}))(1+a/\psi_{\lambda}(\mathfrak{b}))=\mathfrak{b}(1+a/\psi_{\lambda}(\mathfrak{b}))$$
where $\mathrm{ord}_v(a/\psi_{\lambda}(\mathfrak{b}))\geqslant \mathrm{ord}_v(\mathfrak{f})$ for each prime $v\mid \mathfrak{f}$, so that
$$\psi_\lambda(\psi_{\lambda}(\mathfrak{b})+a)=\psi_\lambda(\mathfrak{b})(1+a/\psi_\lambda(\mathfrak{b}))=\psi_\lambda(\mathfrak{b})+a.$$
Hence
$$L_S(\overline{\psi}_{\lambda},s)=\sum\limits_{\mathfrak{b}\in\mathcal{B}}\sum\limits_{a\in\mathfrak{g}}\frac{\overline{\psi_{\lambda}(\mathfrak{b})+a}}{|\psi_{\lambda}(\mathfrak{b})+a|^{2s}}=\sum\limits_{\mathfrak{b}\in\mathcal{B}}H\left(\psi_\lambda(\mathfrak{b}),s,\mathfrak{g}\right).$$
We can renormalise the right hand side to obtain the result.
\end{proof}
The following is a well-known fact from, for example, \cite[Theorem 2.1]{gol-sch}.
\begin{fact}\label{fact1} For all $\mathfrak{b}\in \mathcal{B}$, we have
$$\mathcal{E}_1^*\left(\frac{\Omega_{\lambda}}{g}, \mathcal{L}_{\lambda}\right)\in K(\mathfrak{g})$$
and
\begin{equation}\label{eq2.2.1}\mathcal{E}_1^*\left(\frac{\Omega_{\lambda}}{g}, \mathcal{L}_{\lambda}\right)^{\sigma_\mathfrak{b}}=\mathcal{E}_1^*\left(\frac{\psi(\mathfrak{b})\Omega_{\lambda}}{g}, \mathcal{L}_{\lambda}\right).
\end{equation}
\end{fact}

Now, we concentrate on the case where $E(\lambda)$ is a quadratic twist of $E$. 
\begin{defn}\label{def2.2.3} We say a rational prime $p$ is a special split prime if $p$ splits completely in $L=K(x(E[4]))$, the field obtained by adjoining to $K$ the $x$-coordinates of all non-zero points in $E[4]$.
\end{defn}
\noindent
In addition, it can be shown that a rational prime $p$ is a special split prime if and only if it splits in $K$, and $\psi(\mathfrak{p})\equiv \pm 1\bmod 4$ for both of the primes $\mathfrak{p}$ of $K$ above $p$.  Moreover, $L=K(\boldsymbol{\mu}_{4}, \sqrt[3]{2})$ (see Lemma \ref{lem4} of Appendix A).

For the remainder of this section, we assume that $D\in \mathcal{O}_K$ is such that $D\equiv 1\bmod 3$ and $(D)=\mathfrak{p}_1\cdots \mathfrak{p}_n$ is a square-free product of prime ideals $\mathfrak{p}_j$ of $K$ above special split primes. In addition, we pick the sign $\pi_j$ of the generator of $\mathfrak{p}_j$ so that $\pi_j\equiv 1\bmod 4$, and set $D=\pi_1\cdots \pi_n$ and $S=\{\pi_1,\ldots ,\pi_n\}$. The sign will not matter since we are most interested in the case when $D$ is an integer. Given $\alpha=(\alpha_1,\ldots \alpha_n)$ with $\alpha_j\in \{0, 1\}$ for all $j=1,\ldots , n$, let $D_\alpha\in K$ be of the form $D_\alpha=\pi_1^{\alpha_1}\cdots \pi_n^{\alpha_n}$. Note that for any integers $k_j\geqslant0$ and $D_{\alpha '}=\pi_1^{\alpha_1+2 k_1}\cdots \pi_n^{\alpha_n+2 k_n}$,  we have
$$E(D_\alpha^{3})\cong E(D_{\alpha'}^{3})$$
over $K$, hence we may consider $\alpha=(\alpha_1,\ldots , \alpha_n)\in\{0, 1\}^n$ as an element of $(\mathbb{Z}/2\mathbb{Z})^n$. Given $\alpha\in (\mathbb{Z}/2\mathbb{Z})^n$, let $n_\alpha$ be the number of primes dividing $D_\alpha$ and define $S_\alpha=\{\pi_j:\pi_j\mid D_\alpha\}$.

Let $C(A/\mathbb{Q})$ be the conductor of an elliptic curve $A$ over $\mathbb{Q}$. Recall that if $\mathrm{End}_{\overline{\mathbb{Q}}}(A)\otimes_{\mathbb{Z}}\mathbb{Q}=K$, an imaginary quadratic field, we have
\begin{equation}\label{f1}C(A/\mathbb{Q})=\mathrm{N}_{K/\mathbb{Q}}\mathfrak{f}_A\cdot d_K,
\end{equation}
where $\mathfrak{f}_A$ is the conductor of $\psi_{A/K}$ and $d_K$ is the absolute value of the discriminant of $K/\mathbb{Q}$. In particular, $C(E/\mathbb{Q})=27$, and so the conductor of $\psi$ is $3\mathcal{O}_K$. It can be verified using this result and Tate's algorithm that the conductor of $\psi_{D^3}$ is $\mathfrak{f}=3D\mathcal{O}_K$. It follows that $K\left(E(D^3)_{\mathfrak{f}}\right)$ is isomorphic to  $K(\mathfrak{f})$, the ray class field of $K$ modulo $\mathfrak{f}$. Hence the Artin map gives an isomorphism
$$\left(\mathcal{O}_K/3D\mathcal{O}_K\right)^\times/\widetilde{\boldsymbol{\mu}}_6 \xrightarrow{\sim} \mathrm{Gal}\left(K\left(E(D^3)_{\mathfrak{f}}\right)/K\right)$$
where $\widetilde{\boldsymbol{\mu}}_6$ denotes the image of $\boldsymbol{\mu}_K=\boldsymbol{\mu}_6$ under reduction modulo $\mathfrak{f}$.
Note that since $3$ and $D$ are coprime and $3$ ramifies in $K$, we have an exact sequence
$$0\to \left(\mathcal{O}_K/D\mathcal{O}_K\right)^\times\to \left(\mathcal{O}_K/3D\mathcal{O}_K\right)^\times/\widetilde{\boldsymbol{\mu}}_6\to  \left(\mathcal{O}_K/3\mathcal{O}_K\right)^\times/\boldsymbol{\mu}_6 \to 0,$$
so that $\left(\mathcal{O}_K/3D\mathcal{O}_K\right)^\times/\widetilde{\boldsymbol{\mu}}_6 \cong \left(\mathcal{O}_K/D\mathcal{O}_K\right)^\times.$

Setting $s=1$  and $g=3D$ in Lemma \ref{lem2.2.1} and applying \eqref{eq2.2.1} immediately yields:
\begin{cor}\label{cor2.2.4} For any $\alpha \in (\mathbb{Z}/2\mathbb{Z})^n$, we have
$$\frac{3D}{\Omega_{D_\alpha^3}}L_S(\overline{\psi}_{D_\alpha^3}, 1)=\mathrm{Tr}_{K(\mathfrak{f})/K}\left(\mathcal{E}_1^*\left(\frac{\Omega_{D_\alpha^3}}{3D}, \mathcal{L}_{D_\alpha^3}\right)\right).$$ 
\end{cor}

We wish to find $\mathrm{ord}_2\left(L^{(\text{alg})}(\overline{\psi}_{D^3},1)\right)$. In order to do this, we consider the following sum of imprimitive Hecke $L$-series.
\begin{defn}\label{def2.2.5}
Let
$$\Phi_{D^3}=\sum\limits_{\alpha\in (\mathbb{Z}/2\mathbb{Z})^n}\frac{L_S(\overline{\psi}_{D_\alpha^3}, 1)}{\Omega.}$$
\end{defn}

Using Corollary \ref{cor2.2.4}, we can write this sum in the following way.

\begin{thm}\label{thm2.1}We have
$$\Phi_{D^3}=2^n\mathrm{Tr}_{K(\mathfrak{f})/\mathcal{J}}\left(\frac{1}{3D}\mathcal{E}_1^*\left(\frac{\Omega}{3D}, \mathcal{L}\right)\right),$$
where $\mathcal{J}=\mathbb{Q}\left(\sqrt{-3}, \sqrt{\pi_1}, \ldots , \sqrt{\pi_n}\right)$. 
\end{thm}

\begin{proof} We have for any $\alpha \in (\mathbb{Z}/2\mathbb{Z})^n$,
$$\frac{L_S(\overline{\psi}_{D_\alpha^3}, 1)}{\Omega_{D_\alpha^3}}=\frac{1}{3D}\sum\limits_{\mathfrak{b}\in\mathcal{B}}\mathcal{E}_1^*\left(\frac{\Omega_{D_\alpha^3}}{3D}, \mathcal{L}_{D_\alpha^3}\right)^{\sigma_\mathfrak{b}}$$
and $\Omega_{D_\alpha^3}=\frac{1}{D_\alpha^{1/2}}\Omega$, so
\begin{equation}\label{eq5}\frac{L_S(\overline{\psi}_{D_\alpha^3}, 1)}{\Omega}=\frac{1}{3D}\sum\limits_{\mathfrak{b}\in\mathcal{B}}(D_\alpha^{3})^{\frac{\sigma_{\mathfrak{b}}-1}{6}}\mathcal{E}_1^*\left(\frac{\Omega}{3D}, \mathcal{L}\right)^{\sigma_\mathfrak{b}}
\end{equation}
and
$$(D_\alpha^{3})^{\frac{\sigma_{\mathfrak{b}}-1}{6}}=\left(\frac{D_\alpha}{\mathfrak{b}}\right)_2\in\{\pm 1\},$$
where $\left(\frac{\phantom{a}}{\phantom{a}}\right)_2$ denotes the quadratic residue symbol.  Let $\epsilon_2(\cdot , \mathfrak{b}):  (\mathbb{Z}/2\mathbb{Z})^n\to \{\pm 1\}$ be the $1$-dimensional character defined by $\epsilon_2(\alpha , \mathfrak{b})= \left(\frac{D_\alpha}{\mathfrak{b}}\right)_2$.
Since any $1$-dimensional character is irreducible, considering its inner product with the trivial character gives
$$
\sum\limits_{\alpha\in (\mathbb{Z}/2\mathbb{Z})^n}\epsilon_2(\alpha, \mathfrak{b}) = \left\{ \begin{array}{ll}
 2^n &\mbox{ if $\left(\frac{D_\alpha}{\mathfrak{b}}\right)_2=1$ for all $\alpha\in (\mathbb{Z}/2\mathbb{Z})^n$} \\
 0 &\mbox{ otherwise.}
       \end{array} \right.
$$
Note that $\left(\frac{D_\alpha}{\mathfrak{b}}\right)_2=1$ for all $\alpha\in (\mathbb{Z}/2\mathbb{Z})^n$ if and only if $\left(\frac{\pi_j}{\mathfrak{b}}\right)_2=1$ for all $j=1,\ldots , n$. The result now follows by noting that $\left(\frac{\pi_j}{\mathfrak{b}}\right)_2=1$ for all $j=1,\ldots , n$ if and only if $\sigma_{\mathfrak{b}}\in\mathrm{Gal}(K(\mathfrak{f})/\mathcal{J})$ where $\mathcal{J}=\mathbb{Q}\left(\sqrt{-3}, \sqrt{\pi_1}, \ldots , \sqrt{\pi_n}\right)$. 
\end{proof}

We now make an explicit choice of $\mathcal{B}$. 

\begin{defn} 
Let $\mathcal{C}$ be a set of elements of $\mathcal{O}_K$ such that $c\in \mathcal{C}$ implies $-c\in \mathcal{C}$ and $c \bmod D$ runs over $\left(\mathcal{O}_K/D\mathcal{O}_K\right)^\times$ precisely once. Note that this is possible since $(2,D)=1$  by hypothesis. Furthermore, since $\mathrm{Gal}(K(\mathfrak{f})/K)$ is isomorphic to $\left(\mathcal{O}_K/D\mathcal{O}_K\right)^\times$, the Artin symbol $(c, K(\mathfrak{f})/K)$ runs over $\mathrm{Gal}(K(\mathfrak{f})/K)$ precisely once as $c$ varies in $\mathcal{C}$. In addition, we define
$$\mathcal{B}=\{(3c+D) \; : \; c\in \mathcal{C}\}$$
so that $3c+D\equiv 1 \bmod 3\mathcal{O}_K$ for all $c\in \mathcal{C}$ since $D\equiv 1\bmod 3$ by assumption. In particular, if $\mathfrak{b}=(3c+D)$ then we have  $\psi(\mathfrak{b})=3c+D$ since the conductor of $\psi$ is $3\mathcal{O}_K$. Finally, let
 $$V=\{c\in \mathcal{C} \;\; : \left(\frac{\pi_j}{\mathfrak{b}}\right)_2=1\ \text{ for all } j=1,\ldots, n, \text{ where }\mathfrak{b}=(3c+D)\},$$
where $\left(\frac{\phantom{a}}{\phantom{a}}\right)_2$ denotes the quadratic residue symbol. 
\end{defn}

Note that if $c\in V$ implies $-c\in V$ since
\begin{align*}\left(\frac{\pi_j}{\mathfrak{b}}\right)_2&= \left(\frac{3c+D}{\pi_j}\right)_2\ \mbox{\hspace{20pt} (since $\pi_j\equiv 1\bmod 4$)}\\
&= \left(\frac{3c}{\pi_j}\right)_2\\
&= \left(\frac{-3c}{\pi_j}\right)_2  \mbox{\hspace{40pt} (since $\left(\frac{-1}{\pi_j}\right)_2=1$)}.
\end{align*}

It is clear that we can also write Theorem \ref{thm2.1} in the following way.

\begin{cor} We have
 $$\Phi_{D^3}=2^n\sum\limits_{c\in V} \frac{1}{3D}\mathcal{E}_1^*\left(\frac{c\O}{D}+\frac{\Omega}{3}, \mathcal{L}\right).$$
\end{cor}

Using the relation between the Eisenstein series and the Weierstrass $\wp$-function, we can show:

\begin{thm}\label{thm2.2.9}
We have
$$\sum\limits_{c\in V}\mathcal{E}_1^*\left(\frac{c\O}{D}+\frac{\Omega}{3}, \mathcal{L}\right)=\frac{1}{2}\left(\sum\limits_{c\in V}\frac{9}{3-\wp\left(\frac{c\O}{D},\mathcal{L}\right)}\right)-\#(V).$$
\end{thm}

\begin{proof}
Let $$s_2(\mathcal{L})=\lim_{\substack{s\to 0 \\ s>0}}\sum_{w\in \mathcal{L}\backslash \{0\}}w^{-2}|w |^{-2s}.$$
Then by \cite[Proposition 1.5]{gol-sch}, we have
$$\mathcal{E}_1^*(z,\mathcal{L})=\zeta(z,\mathcal{L})-zs_2(\mathcal{L})-\overline{z}A(\mathcal{L})^{-1}.$$
Here, $\zeta(z,\mathcal{L})$ is the Weierstrass zeta function of $\mathcal{L}$ and $A(\mathcal{L}):=\frac{\overline{u}v-u\overline{v}}{2\pi i}$ where $(u,v)$ is a base of $\mathcal{L}$ over $\mathbb{Z}$ satisfying $\mathrm{Im}(v/u)>0$. Thus we have $A(\mathcal{L})=\frac{\Omega^2(\omega-\overline{\omega})}{2\pi i}=\frac{\sqrt{3}\Omega^2}{2\pi}$, and we can see that $s_2(\mathcal{L})=0$ on noting that $\omega\in \mathcal{L}$ which gives $\omega^{-2} s_2(\mathcal{L})=s_2(\mathcal{L})$. Hence
$$\mathcal{E}_1^*(z, \mathcal{L})=\zeta(z,\mathcal{L})-\frac{2\pi\overline{z}}{\sqrt{3}\Omega^2}.$$
Recall also that for $z_1, z_2\in \mathbb{C}$, we have an addition formula:
$$\zeta(z_1+z_2, \mathcal{L})=\zeta(z_1, \mathcal{L})+\zeta(z_2, \mathcal{L})+\frac{1}{2}\frac{\wp'(z_1, \mathcal{L})-\wp'(z_2, \mathcal{L})}{\wp(z_1, \mathcal{L})-\wp(z_2, \mathcal{L})}.$$
Applying this with $z_1=\frac{\Omega}{3}$, $z_2=\frac{c\O}{D}$, we get
\begin{align*}\sum\limits_{c\in V}\mathcal{E}_1^*&\left(\frac{c\O}{D}+\frac{\Omega}{3}, \mathcal{L}\right)=\sum\limits_{c\in V}\left(\zeta\left(\frac{c\O}{D}+\frac{\Omega}{3}, \mathcal{L}\right)-\left(\frac{\overline{c}\Omega}{\overline{D}}+\frac{\Omega}{3}\right)\frac{2\pi}{\sqrt{3}\Omega^2}\right)\\
&=\sum\limits_{c\in V}\left(\zeta\left(\frac{\Omega}{3}, \mathcal{L}\right)+\zeta\left(\frac{c\O}{D}, \mathcal{L}\right)+\frac{1}{2}\frac{\wp'(\frac{\Omega}{3}, \mathcal{L})-\wp'(\frac{c\O}{D}, \mathcal{L})}{\wp(\frac{\Omega}{3}, \mathcal{L})-\wp(\frac{c\O}{D}, \mathcal{L})}-\left(\frac{\overline{c}\Omega}{\overline{D}}+\frac{\Omega}{3}\right)\frac{2\pi}{\sqrt{3}\Omega^2}\right).
\end{align*}
Next, we use the key property that, if $c\in V$, then also $-c\in V$. Since $\zeta(z,\mathcal{L})$ and $\wp'(z,\mathcal{L})$ are odd functions, and $\wp(z,\mathcal{L})$ is an even function, it follows that
\begin{align*}
\sum\limits_{c\in V}\mathcal{E}_1^*\left(\frac{c\O}{D}+\frac{\Omega}{3}, \mathcal{L}\right)&=\left(\sum\limits_{c\in V}\frac{1}{2}\frac{\wp'\left(\frac{\Omega}{3},\mathcal{L}\right)}{\wp\left(\frac{\Omega}{3},\mathcal{L}\right)-\wp\left(\frac{c\O}{D},\mathcal{L}\right)}\right)+\#(V)\left(\zeta\left(\frac{\Omega}{3},\mathcal{L}\right)-\frac{2\pi}{3\sqrt{3}\Omega}\right).
\end{align*}

By applying formulae (3.2) and (3.3) of \cite[p. 126]{stephens}, we obtain
\begin{equation}\label{eq2.2.4}\zeta(z+1,\mathcal{O}_K)=\zeta(z,\mathcal{O}_K)+\frac{2\pi}{\sqrt{3}}, \;\;\;\zeta(z+\omega,\mathcal{O}_K)=\zeta(z,\mathcal{O}_K)+\frac{2\pi}{\sqrt{3}}\overline{\omega}.
\end{equation}
Letting $z=-\frac{1}{3}$ in \eqref{eq2.2.4} gives
$$\zeta\left(\frac{2}{3}, \mathcal{O}_K\right)+\zeta\left(\frac{1}{3}, \mathcal{O}_K\right)=\frac{2\pi}{\sqrt{3}}.$$
But we have $\zeta\left(\Omega z, \mathcal{L}\right)=\frac{1}{\Omega}\zeta\left(z, \mathcal{O}_K\right)$, so
\begin{equation}\label{eq2.2.5}\zeta\left(\frac{2\O}{3}, \mathcal{L}\right)+\zeta\left(\frac{\Omega}{3}, \mathcal{L}\right)=\frac{2\pi}{\sqrt{3}\Omega}.
\end{equation}
On the other hand, we have
$$\zeta(2z, \mathcal{L})=2\zeta(z,\mathcal{L})+\frac{\wp''(z, \mathcal{L})}{2\wp'(z,\mathcal{L})},$$
and by differentiating the equation $\wp'(z,\mathcal{L})^2=4\wp(z,\mathcal{L})^3-3^3$, we get $\wp''(z,\mathcal{L})=6\wp(z,\mathcal{L})^2$. Also, by computation we get
$$\wp\left(\frac{\Omega}{3},\mathcal{L}\right)=3, \;\; \wp'\left(\frac{\Omega}{3},\mathcal{L}\right)=9,$$ thus
\begin{equation}\label{eq2.2.6}\zeta\left(\frac{2\O}{3}, \mathcal{L}\right)-2\zeta\left(\frac{\Omega}{3}, \mathcal{L}\right)=\frac{\wp''\left(\frac{\Omega}{3}, \mathcal{L}\right)}{2\wp'\left(\frac{\Omega}{3},\mathcal{L}\right)}=\frac{6\wp^2\left(\frac{\Omega}{3}, \mathcal{L}\right)}{2\wp'\left(\frac{\Omega}{3},\mathcal{L}\right)}=3.
\end{equation}
Now, solving \eqref{eq2.2.5} and \eqref{eq2.2.6} gives
$$\zeta\left(\frac{\Omega}{3}, \mathcal{L}\right)=\frac{2\pi}{3\sqrt{3}\Omega}-1.$$
Hence
$$\sum\limits_{c\in V}\mathcal{E}_1^*\left(\frac{c\O}{D}+\frac{\Omega}{3}, \mathcal{L}\right)=\left(\sum\limits_{c\in V}\frac{1}{2}\frac{\wp'\left(\frac{\Omega}{3},\mathcal{L}\right)}{\wp\left(\frac{\Omega}{3},\mathcal{L}\right)-\wp\left(\frac{c\O}{D},\mathcal{L}\right)}\right)-\#(V).$$
Substituting the values $\wp\left(\frac{\Omega}{3},\mathcal{L}\right)=3$ and $\wp'\left(\frac{\Omega}{3},\mathcal{L}\right)=9$ again gives the result.
\end{proof}

Now we prove the following integrality result of the Eisenstein series.
\begin{cor}\label{cor1}
For $n\geqslant 1$, we have
$$\mathrm{ord}_2\left(\sum\limits_{c\in V}\mathcal{E}_1^*\left(\frac{c\O}{D}+\frac{\Omega}{3}, \mathcal{L}\right)\right)\geqslant 0.$$ 
\end{cor}

\begin{proof}
Given $c\in V$, let $P$ be the point on $E: y^2=4x^3-3^3$ given by
$$x(P)=\wp\left(\frac{c\O}{D},\mathcal{L}\right), \;\;\; y(P)=\wp'\left(\frac{c\O}{D},\mathcal{L}\right)$$
and define 
$$\mathscr{M}(c,D)=\frac{9}{3-x(P)}.$$
Recall that $E$ has minimal Weierstrass form
$$E: Y^2+Y=X^3-7$$
which has discriminant $3^9$, so $E$ has good reduction at $2$ over $K$. This means that $\mathrm{ord}_2(X(P))\geqslant 0$ since $P$ is a torsion point on $E$ of order prime to $2$. 
Further, $x=X$ in the change of coordinates which gives the minimal Weierstrass form, and so we have
\begin{align*}\mathscr{M}(c,D)=\frac{9}{3-X(P)}.
\end{align*}
We claim that $\mathrm{ord}_2(3-X(P))=0$. Suppose for a contradiction that $\mathrm{ord}_2(3-X(P))>0$. Then let $Q=(3,4)$ be the point on $E$ which we know is a $3$-torsion, so that we have $\mathrm{ord}_2(X(Q)-X(P))>0$. Hence, under reduction modulo $2$, we would have $X(\widetilde{Q})=X(\widetilde{P})$ where $\widetilde{\phantom{a}}$ denotes reduction modulo $2$. Then we have $\widetilde{P}=\pm \widetilde{Q}$, so either $P-Q$ or $P+Q$ is in the kernel of the reduction map, so it must correspond to an element in the formal group of $E$ at $2$, and therefore its order must be a power of $2$. But this is not possible since $P$ has order $D$ and $Q$ has order $3$, both of which are coprime to $2$. Hence 
\begin{align*}
 \mathrm{ord}_2(\mathscr{M}(c,D))&=\mathrm{ord}_2(9)-\mathrm{ord}_2(3-X(P))\\
&=0.
\end{align*}
But $\mathscr{M}(c,D)=\mathscr{M}(-c,D)$ since $\wp(z)$ is an even function and $\#(V)$ is even, so
   \begin{align*}
    \mathrm{ord}_2(\sum\limits_{c\in V}\mathscr{M}(c,D))&\geqslant 1.
   \end{align*}
It follows that 
\begin{align*}
\mathrm{ord}_2\left(\sum\limits_{c\in V}\mathcal{E}_1^*\left(\frac{c\O}{D}+\frac{\Omega}{3}, \mathcal{L}\right)\right)&=\min\left(\mathrm{ord}_2\left(\frac{1}{2}\sum\limits_{c\in V}\mathscr{M}(c,D)\right), \mathrm{ord}_2\left(\#(V)\right)\right)\\
&\geqslant 0
\end{align*}
as required.
\end{proof}

\begin{remark}\label{rem1}
For $n=0$ (i.e. for $E$), a computation using Magma gives
$$L^{(\text{alg})}(\overline{\psi},1)=\frac{1}{3}.$$
\end{remark}
Thus we have proved:
\begin{thm}\label{thm2.2.12} Let $D\in \mathcal{O}_K$ be as above and let $n$ be the number of primes in $\mathcal{O}_K$ dividing $D$. Then we have
$$\mathrm{ord}_2(\Phi_{D^3})\geqslant n.$$
\end{thm}

Finally, we are ready to prove the first main result:
\begin{thm} Let $D\in \mathcal{O}_K$ be as above and let $n$ be the number of primes in $\mathcal{O}_K$ dividing $D$. Then 
 $$\mathrm{ord}_2\left(L^{(\text{alg})}(\overline{\psi}_{D^3},1)\right)\geqslant n.$$
\end{thm}

\begin{proof}
 We prove this by induction on $n$. Write $D=D_\alpha$, and given $\alpha,\beta\in (\mathbb{Z}/2\mathbb{Z})^n$, we write $\beta< \alpha$ if $D_\beta\mid D_\alpha$ but $D_\beta\neq D_\alpha$. If $n_\alpha=1$, $S_\alpha=\{\pi_1\}$ say, then
$$\Phi_{\pi_1^3}=\frac{L_{S_\alpha}(\overline{\psi},1)}{\Omega}+\frac{L(\overline{\psi}_{\pi_1^3},1)}{\Omega}.$$
By Theorem \ref{thm2.2.12}, we know that $\mathrm{ord}_2(\Phi_{\pi_1^3})\geqslant 1$. Now,
\begin{align*}\frac{L_{S_\alpha}(\overline{\psi},1)}{\Omega}&=\left(1-\frac{\overline{\psi}((\pi_1))}{\pi_1\overline{\pi}_1}\right)\frac{L(\overline{\psi},1)}{\Omega}\\
&=\left(\frac{\pi_1\pm1}{\pi_1}\right)\frac{1}{3}
\end{align*}
since $\psi((\pi_1))=\pm\pi_1$ and by Remark \ref{rem1} we have $\frac{L(\overline{\psi},1)}{\Omega}=\frac{1}{3}$.
But $\mathrm{ord}_2\left(\frac{\pi_1\pm1}{\pi_1}\right)\geqslant 1$, hence
\begin{align*}
 \mathrm{ord}_2\left(\frac{L(\overline{\psi}_{\pi_1^3},1)}{\Omega}\right)\geqslant 1=n_\alpha.
\end{align*}
Now suppose $n_\alpha>1$ and our result holds for $0< \beta < \alpha$. Again,
$$\Phi_{D_\alpha^3}=\frac{L_{S_\alpha}(\overline{\psi},1)}{\Omega}+\sum\limits_{0< \beta < \alpha}\frac{L_{S_\alpha}(\overline{\psi}_{D_\beta^3},1)}{\Omega}+\frac{L_{S_\alpha}(\overline{\psi}_{D_\alpha^3},1)}{\Omega},$$
where the last term is primitive. We know by Theorem \ref{thm2.2.12} that $\mathrm{ord}_2(\Phi_{D_\alpha^3})\geqslant n_\alpha$.
Now
\begin{align*}
 \frac{L_{S_\alpha}(\overline{\psi},1)}{\Omega}&=\prod\limits_{\pi\in S_\alpha}\left(1-\frac{\overline{\psi}((\pi))}{\pi\overline{\pi}}\right)\frac{L(\overline{\psi},1)}{\Omega}\\
&=\prod\limits_{\pi\in S_\alpha}\left(\frac{\pi\pm1}{\pi}\right)\frac{1}{3}
\end{align*}
where $\mathrm{ord}_2\left(\frac{\pi\pm 1}{\pi}\right)\geqslant1$ for each $\pi\in S_\alpha$. Hence
\begin{align*}\mathrm{ord}_2\left(\frac{L_{S_\alpha}(\overline{\psi},1)}{\Omega}\right)&\geqslant \#(S_\alpha)\\
&\geqslant n_\alpha.
\end{align*}
Also for $0< \beta< \alpha$,
$$\frac{L_{S_\alpha}(\overline{\psi}_{D_\beta^3},1)}{\Omega}=\prod\limits_{\pi\in S_\alpha\backslash S_\beta}\left(1-\frac{\overline{\psi}_{D_\beta^3}((\pi))}{\pi\overline{\pi}}\right)\frac{L(\overline{\psi}_{D_\beta^3},1)}{\Omega}.$$
We have $\psi_{D_\beta^3}((\pi))=\left(\frac{D_\beta}{\pi}\right)_6^3\psi((\pi))=\pm \pi$. Hence 
\begin{align*}
\mathrm{ord}_2\left(\prod\limits_{\pi\in S_\alpha\backslash S_\beta}\left(1-\frac{\overline{\psi}_{D_\beta^3}((\pi))}{\pi\overline{\pi}}\right)\right)&=\mathrm{ord}_2\left(\prod\limits_{\pi\in S_\alpha\backslash S_\beta}\left(\frac{\pi\pm 1}{\pi}\right)\right)\\
&\geqslant \#(S_\alpha\backslash S_\beta)\\
&=n_\alpha-n_\beta.
\end{align*}
Furthermore, by the induction hypothesis, $\mathrm{ord}_2\left(\frac{L(\overline{\psi}_{D_\beta^3},1)}{\Omega}\right)\geqslant n_\beta$. Thus
\begin{align*} 
 \mathrm{ord}_2\left(\frac{L_{S_\alpha}(\overline{\psi}_{D_\beta^3},1)}{\Omega}\right)&\geqslant (n_\alpha-n_\beta)+n_\beta\\
&=n_\alpha,
\end{align*}
and so 
$$\mathrm{ord}_2\left(\sum\limits_{0< \beta < \alpha}\frac{L_{S_\alpha}(\overline{\psi}_{D_\beta^3},1)}{\Omega}\right)\geqslant n_\alpha.$$
It follows that 
$$\mathrm{ord}_2\left(\frac{L(\overline{\psi}_{D_\alpha^3},1)}{\Omega}\right)\geqslant n_\alpha$$
as required.
\end{proof}

Recalling $L(E(\lambda),1)=L(\overline{\psi}_\lambda,1)$, the following is an immediate consequence.
\begin{thm}\label{thm2.2.14}Let $D>1$ be an integer which is a product of $k(D)$ distinct special split primes. Then
 $$\mathrm{ord}_2\left(L^{(\text{alg})}(E({D^3}),1)\right)\geqslant 2k(D).$$
\end{thm}

\begin{remark}\label{rem2} The bound obtained in Theorem \ref{thm2.2.14} is sharp. For example, let $\pi$ be the prime $13 + 12\omega$ and let $D=\rm{N}(\pi)=157$, which is a rational prime. Then $L^{(\text{alg})}(E({D^3}),1)=12$ so $\mathrm{ord}_2\left(L^{(\text{alg})}(E({D^3}),1)\right)=2$, as required. More numerical examples can be found in Appendix B.
\end{remark}

\section{Cubic Twists.}\label{section2.3}
Now we look at the cubic twists of $E$, i.e. the curves of the form
$$E(D^2): y^2=4x^3-3^3D^2$$
for a cube-free integer $D$. This is isomorphic to the curve
 $$Y^2+DY=X^3-7D^2$$
 via the change of variables $X=x$ and $Y=2y+D$. Let $\psi_{D^2}$ denote the Gr\"{o}ssencharacter of $E(D^2)/K$. 
\begin{defn} We say a prime $\pi$ of $K$ is \emph{cubic-special} if it splits completely in the field $K(E[27])$, but does not split completely in the strictly larger field $K(E[27])((1-\omega)^{1/9})$.
\end{defn}

The following characterisation of cubic-special primes will be useful, in particular in proving Corollary \ref{cor6} of Appendix A.

\begin{lem}\label{lem3} A prime $\pi$ of $K$ is cubic special if and only if $\pi\equiv 1\bmod 27$ and $9$ divides the order of $1-\omega$ in $\left(\mathcal{O}_K/\pi\mathcal{O}_K\right)^\times$. The set consisting of such primes has density $\frac{2}{3}$ in the set of primes of $K$ congruent to $1$ modulo $27$. In particular, there are infinitely many such primes.
\end{lem}

\begin{proof} First, we note that $K(E[27])$ is equal to the ray class field $K(27)$ of $K$ modulo $27$  by \cite[Lemma 3]{coa-wil}. Since $\mathbb{Q}(\boldsymbol{\mu}_{27})\subset K(27)$, it follows that $K(27)\left((1-\omega)^{\frac{1}{9}}\right)/ K(27)$ is a Galois extension. Also $K(27)\left((1-\omega)^{\frac{1}{9}}\right)/K$ is not an abelian extension, since its subextension $K\left((1-\omega)^{\frac{1}{9}}\right)/K$ is not Galois. In addition, $K(27)\left((1-\omega)^{\frac{1}{9}}\right)/K(27)$ is a degree $3$ extension since we showed that $\left(\frac{1-\omega}{\pi}\right)_3=1$, i.e. $(1-\omega)^{\frac{1}{3}}\in K(27)$. Let $H$ denote the Galois group of this degree $3$ extension. Furthermore, let $G$ denote the Galois group $\mathrm{Gal}\left(K(27)\left((1-\omega)^{\frac{1}{9}}\right)/K\right)$, and let $\mathrm{Frob}_\pi\in G$ denote the Frobenius at $\pi$. Then $ \mathrm{Frob}_\pi|_{K(27)}=id$ in $H$ if and only if $\psi_{E(\pi^2)/K}\left((\pi)\right)\equiv 1\bmod 27$. If we take a prime $\pi$ such that $\mathrm{Frob}_\pi\in H\backslash \{id\}$, then $(1-\omega)$ is not a ninth power modulo $\pi$ in $K(27)\left((1-\omega)^{\frac{1}{9}}\right)$, and it follows that the order of  $1-\omega$ must be divisible by $9$ since $27$ divides $\mathrm{N}(\pi)-1=|\left(\mathcal{O}_K/\pi\mathcal{O}_K\right)^\times|$. By the \v{C}ebotarev density theorem, the density of such primes is $\frac{2}{3}$. 
\end{proof}

From now on, let us assume that each prime $\pi$ of $K$ dividing $D$ is cubic-special.
Note that if $p$ is a rational prime such that $p\equiv 1\bmod 3$, then $p$ always splits in $K$ since we can write $p=a^2-ab+b^2=(a+b\omega)(a+b\overline{\omega})$ for some integers $a$ and $b$. In addition, if $p\equiv 1\bmod 27$, it can easily be shown that we can assume $b\equiv 0\bmod 27$ and $a\equiv 1\bmod 27$ using symmetry in $a$ and $b$ and change of sign of $a$. Hence we can write $p=\pi\overline{\pi}$ with $\pi\in\mathcal{O}_K$ and $\pi\equiv 1\bmod 27$.

Before we begin, it will be useful to find a model for our curve $E: Y^2 + Y = X^3 - 7$ where $E$ has good reduction at $3$. Let  $u=\frac{\sqrt{\alpha}}{\beta^2}$ where $\alpha=\frac{27+3\sqrt{-3}}{2}$, $\beta=\sqrt[3]{\frac{1-3\sqrt{-3}}{2}}$, and let \mbox{$r=-\frac{3}{2}\sqrt[3]{\frac{-13-3\sqrt{-3}}{2}}$}. Then the change of variables $x=u^2X+r$, $y=2u^3Y$,  gives an equation for $E$ with good reduction at $3$ (see Proposition \ref{prop4} of Appendix A).

Given $\alpha=(\alpha_1,\ldots \alpha_n)$ with $\alpha_j\in \{0, 1, 2\}$ for all $j=1,\ldots , n$, let $D_\alpha$ be an element of $K$ of the form $D_\alpha=\pi_1^{\alpha_1}\cdots \pi_n^{\alpha_n}$ where $\pi_j$ are distinct cubic-special primes. Similarly to the quadratic twist case, we may consider $\alpha=(\alpha_1,\ldots , \alpha_n)\in\{0, 1, 2\}^n$ as an element of $(\mathbb{Z}/3\mathbb{Z})^n$. Given $\alpha\in (\mathbb{Z}/3\mathbb{Z})^n$, let $n_\alpha$ be the number of distinct primes of $K$ dividing $D_\alpha$ and define $S_\alpha=\{\pi_j:\pi_j\mid D_\alpha\}$. Pick $\alpha \in (\mathbb{Z}/3\mathbb{Z})^n$ such that $n_\alpha=n$, and set $D=D_\alpha$ and $S=\{\pi_1,\ldots \pi_n\}$. We will study the following sum of imprimitive Hecke $L$-functions (see Definition \ref{def2.2.5}).
\begin{defn}\label{def2} Given $D$ as above, let
$$\Phi_{D^2}=\sum\limits_{\alpha\in (\mathbb{Z}/3\mathbb{Z})^n}\frac{L_S(\overline{\psi}_{D_\alpha^2}, 1)}{\Omega}.$$
\end{defn}

Let $\mathfrak{f}$ be the conductor of the Gr\"{o}ssencharacter $\psi_{D^2}$. Then again, a computation using Tate's algorithm shows that \mbox{$\mathfrak{f}=3D\mathcal{O}_K$}. Also, the Artin map gives an isomorphism between $\mathrm{Gal}(K(\mathfrak{f})/K)$ and $\left(\mathcal{O}_K/3D\mathcal{O}_K\right)^\times/\widetilde{\boldsymbol{\mu}}_6$, which is isomorphic to $\left(\mathcal{O}_K/D\mathcal{O}_K\right)^\times$ since $(3, D)=1$ and $3$ ramifies in $K$. Now let $\mathcal{C}$ be a set of elements of $\mathcal{O}_K$ such that $c\in \mathcal{C}$ implies $\omega c$, $\omega^2 c\in \mathcal{C}$ and $c \bmod D$ runs over $\left(\mathcal{O}_K/D\mathcal{O}_K\right)^\times$ precisely once. This is possible since $3$ and $D$ are coprime by assumption. Then let 
$$\mathcal{B}=\{(3c+D) \; : \; c\in \mathcal{C}\}$$
so that $3c+D\equiv 1 \bmod 3\mathcal{O}_K$, where $3\mathcal{O}_K$ is the conductor of $\psi$. In particular, if $\mathfrak{b}=(3c+D)\in \mathcal{B}$ then we have  $\psi(\mathfrak{b})=3c+D$.

Let $m$ be such that $\boldsymbol{\mu}_m\subset K$. For $a\in K^*$ and $\mathfrak{b}$ an ideal of $K$ coprime to $m$ and $a$, we write $\left(\frac{a}{\mathfrak{b}}\right)_m$ for the $m$-th power residue symbol defined by the equation
$$(\sqrt[m]{a})^{\sigma_\mathfrak{b}}=\left(\frac{a}{\mathfrak{b}}\right)_m\sqrt[m]{a},$$
where $\sigma_\mathfrak{b}=\left(\mathfrak{b}, K(\sqrt[m]{a})/K\right)\in \mathrm{Gal}\left( K(\sqrt[m]{a})/K\right)$ denotes the Artin symbol of $\mathfrak{b}$. Also, for any $a, b\in K^*$, we define
$$\left(\frac{a}{b}\right)_m=\prod\limits_{v}\left(\frac{a}{v}\right)_m^{v(b)},$$
where $v$ runs through all primes of $K$ coprime to $a$. Recall also that for a prime $\pi$ of $K$ and $c\in\left(\mathcal{O}_K/\pi\mathcal{O}_K\right)^\times$, we have Euler's criterion
$$\left(\frac{c}{\pi}\right)_m\equiv c^{\frac{\rm{N}(\pi)-1}{m}}\bmod \pi.$$

\begin{defn} Let
 $$V=\{c\in \mathcal{C} \;\; : \left(\frac{\pi_j}{\mathfrak{b}}\right)_3=1\ \text{ for all } j=1,\ldots, n \text{, where } \mathfrak{b}=(3c+D)\}.$$
\end{defn}
Recall that we have $\left(\frac{1-\omega}{\pi_j}\right)_3=\left(\frac{1-\omega^2}{\pi_j}\right)_3=\omega^m$ and $\left(\frac{\omega}{\pi_j}\right)_3=\omega^{-m-n}$ where $m, n\in \mathbb{Z}$ are such that $\pi_j=1+3(m+n\omega)$ (see \cite[p. 354]{cas-fro}). Hence for $c\in V$ we have
\begin{align*} \left(\frac{\pi_j}{\mathfrak{b}}\right)_3 &= \left(\frac{3c+D}{\pi_j}\right)_3 \mbox{\hspace{20pt} (since $\pi_j\equiv \mathfrak{b}\equiv 1\bmod 3$, see \cite[p. 354]{cas-fro})}\\
&= \left(\frac{3c}{\pi_j}\right)_3\\
&=\left(\frac{c}{\pi_j}\right)_3 \;\;\;\;\;\;\;\text{(since $\pi_j\equiv 1\bmod 9$, we have $\left(\frac{1-\omega}{\pi_j}\right)_3=\left(\frac{1-\omega^2}{\pi_j}\right)_3=1$).}
\end{align*}
Furthermore, by assumption on $\pi_j$, we have $m+n\equiv 0\bmod 3$ so $ \left(\frac{\omega}{\pi_j}\right)_3=1$. Hence $\left(\frac{c}{\pi_j}\right)_3=\left(\frac{\omega c}{\pi_j}\right)_3=\left(\frac{\omega^2c}{\pi_j}\right)_3$. So $c\in V$ implies $\omega c$, $\omega^2 c\in V$.

It is also easy to check that
\begin{align*}\mathcal{L}_{D^2}&=\frac{\Omega}{\sqrt[3]{D}}\mathcal{O}_K.
\end{align*}

\begin{thm}\label{thm2.3.5} We have

$$\Phi_{D^2}=3^n\sum\limits_{c\in V} \frac{1}{3D}\mathcal{E}_1^*\left(\frac{c\O}{D}+\frac{\Omega}{3}, \mathcal{L}\right).$$

\end{thm}

\begin{proof} It is clear that Lemma \ref{lem2.2.1}, Fact \ref{fact1} and Corollary \ref{cor2.2.4} still apply. Thus, for any $\alpha\in (\mathbb{Z}/3\mathbb{Z})^n$,
$$\frac{L_S(\overline{\psi}_{D_\alpha^2}, 1)}{\Omega_{D_\alpha^2}}=\frac{1}{3D}\sum\limits_{\mathfrak{b}\in\mathcal{B}}\mathcal{E}_1^*\left(\frac{\Omega_{D_\alpha^2}}{3D}, \mathcal{L}_{D_\alpha^2}\right)^{\sigma_\mathfrak{b}}$$
and $\Omega_{D_\alpha^2}=\frac{1}{D_\alpha^{1/3}}\Omega$, so
\begin{equation}\label{eq3.5}\frac{L_S(\overline{\psi}_{D_\alpha^2}, 1)}{\Omega}=\frac{1}{3D}\sum\limits_{\mathfrak{b}\in\mathcal{B}}(D_\alpha^{2})^{\frac{\sigma_{\mathfrak{b}}-1}{6}}\mathcal{E}_1^*\left(\frac{\Omega}{3D}, \mathcal{L}\right)^{\sigma_\mathfrak{b}}
\end{equation}
and
$$(D_\alpha^{2})^{\frac{\sigma_{\mathfrak{b}}-1}{6}}=\left(\frac{D_\alpha}{\mathfrak{b}}\right)_3\in\boldsymbol{\mu}_3.$$

We have a character $\epsilon_3(\cdot , \mathfrak{b}): \left(\mathbb{Z}/3\mathbb{Z}\right)^n\to \boldsymbol{\mu}_3$ defined by
$\mbox{$\epsilon_3(\alpha , \mathfrak{b})= \left(\frac{D_\alpha}{\mathfrak{b}}\right)_3$}$. This is a $1$-dimensional character, and since any $1$-dimensional character is irreducible, considering its inner product with the trivial character gives
$$\sum\limits_{\alpha\in \left(\mathbb{Z}/3\mathbb{Z}\right)^n} \epsilon_3(\alpha, \mathfrak{b})= \left\{ \begin{array}{ll}
 3^n &\mbox{ if $\left(\frac{D_\alpha}{\mathfrak{b}}\right)_3=1 \;\; \text{ for all } \alpha\in (\mathbb{Z}/3\mathbb{Z})^n$} \\
 0 &\mbox{ otherwise.}
       \end{array} \right.
$$
Note that $\left(\frac{D_\alpha}{\mathfrak{b}}\right)_3=1$ for all $\alpha\in (\mathbb{Z}/3\mathbb{Z})^n$ if and only if $\left(\frac{\pi_j}{\mathfrak{b}}\right)_3=1$ for all $j=1,\ldots , n$.
It follows that
$$\Phi_{D_\alpha^2}=3^n\sum\limits_{c\in V} \frac{1}{3D}\mathcal{E}_1^*\left(\frac{\Omega}{3D}, \mathcal{L}\right)^{\sigma_\mathfrak{b}},$$
where $\mathfrak{b}=3c+D$. Again, applying equation \eqref{eq2.2.1} gives the result.
\end{proof}

As in Theorem \ref{thm2.2.9}, we have
\begin{thm}\label{thm2.3.6}
$$\sum\limits_{c\in V}\mathcal{E}_1^*\left(\frac{c\O}{D}+\frac{\Omega}{3}, \mathcal{L}\right)=\frac{1}{2}\left(\sum\limits_{c\in V}\frac{9-\wp'\left(\frac{c\O}{D},\mathcal{L}\right)}{3-\wp\left(\frac{c\O}{D},\mathcal{L}\right)}\right)-\#(V).$$
\end{thm}

\begin{proof} The proof is almost identical to the proof of Theorem \ref{thm2.2.9}, since the addition formula for $\zeta(z,\mathcal{L})$ implies $\zeta\left(\frac{c\O}{D}, \mathcal{L}\right)+\zeta\left(\frac{\omega c\O}{D}, \mathcal{L}\right)+\zeta\left(\frac{\omega^2 c\O}{D}, \mathcal{L}\right)=0$, and we have $c+\omega c+\omega^2 c=0$ for any $c\in V$.
\end{proof}

This gives:
\begin{cor}\label{cor2.3.7}
For $n\geqslant 1$, we have
$$\mathrm{ord}_3\left(\sum\limits_{c\in V}\mathcal{E}_1^*\left(\frac{c\O}{D}+\frac{\Omega}{3}, \mathcal{L}\right)\right)\geqslant 1.$$ 
\end{cor}

Before we prove this, let us prove:
\begin{prop}\label{prop2.3.8}$\mathrm{ord}_3(\#(V))\geqslant 2$.
\end{prop}

\begin{proof} Given $\alpha_i\in\{0,1,2\}$ for $i=1,\ldots , n$, let
$$V_{(\alpha_1,\ldots \alpha_{n})}=\left\{c\in \mathcal{C}: \left(\frac{c}{\pi_i}\right)=\omega^{\alpha_i} \;\;\text{for all } i\in\{1,\ldots n\}\right\},$$
so that now we have $V=V_{(0,\ldots , 0)}$. Given any $(\alpha_1,\ldots , \alpha_{n})$, if we can find $b\in \mathcal{C}$ such that $\left(\frac{b}{\pi_i}\right)=\omega^{\alpha_i}$, then clearly we can write 
\begin{align*}V_{(\alpha_1,\ldots \alpha_{n})}&=bV\\
&=\{bc: c\in V\}
\end{align*}
and if there is no such $b$, then $V_{(\alpha_1,\ldots \alpha_{n})}=\emptyset$.
Also, we have 
$$\mathcal{C}=\bigcup_{(\alpha_1,\ldots , \alpha_n)\in\{0,1,2\}^n} V_{(\alpha_1,\ldots , \alpha_n)},$$
so 
\begin{align*}\#(\mathcal{C})=k\#(V)
\end{align*}
for some positive integer $k\leqslant 3^n$, so that $\mathrm{ord}_3(k)\leqslant n$. On the other hand, $\mathrm{ord}_3\left(\#(\mathcal{C})\right)=\mathrm{ord}_3\left((\rm{N}(\pi_1)-1)\cdots (\rm{N}(\pi_n)-1)\right)\geqslant 3n$. Hence, $\mathrm{ord}_3(\#(V))\geqslant 3n-n=2n\geqslant 2$ for $n\geqslant 1$, so $9\mid \#(V)$ as required.
\end{proof}

Now we are ready to prove Corollary \ref{cor2.3.7}.

\begin{proof} (of Corollary \ref{cor2.3.7})
Let $P$ be the point on $E: y^2=4x^3-3^3$ given by
$$x(P)=\wp\left(\frac{c\O}{D},\mathcal{L}\right), \;\;\; y(P)=\wp'\left(\frac{c\O}{D},\mathcal{L}\right),$$
and define 
$$\mathscr{M}(c,D)=\frac{9-y(P)}{3-x(P)}.$$

Now, write $V$ as a union $H\cup \omega H \cup \omega^2H$ for some set $H$. Then
$$\sum\limits_{c\in V}\mathscr{M}(c,D)=\sum\limits_{c\in H}\frac{9-\wp'\left(\frac{c\O}{D},\mathcal{L}\right)}{3-\wp\left(\frac{c\O}{D},\mathcal{L}\right)}+\frac{9-\wp'\left(\frac{\omega c\O}{D},\mathcal{L}\right)}{3-\wp\left(\frac{\omega c\O}{D},\mathcal{L}\right)}+\frac{9-\wp'\left(\frac{\omega^2 c\O}{D},\mathcal{L}\right)}{3-\wp\left(\frac{\omega^2c\O}{D},\mathcal{L}\right)}.$$
Recall that  $E$ has complex multiplication by $\omega$ via $\omega(x,y)=(\omega x,y)$, so $\wp'(\frac{\omega^i c\O}{D},\mathcal{L})=\wp'(\frac{c\O}{D},\mathcal{L})$ for $i=0, 1, 2$. Moreover, $\mathcal{L}=\omega \mathcal{L}$ so $\wp\left(\frac{\omega^ic\O}{D},\mathcal{L}\right)=\wp\left(\frac{\omega^ic\O}{D},\omega^i\mathcal{L}\right)$, and $\wp$ is homogeneous of degree $-2$ so this simplifies to

\begin{align*}\sum\limits_{c\in V}\mathscr{M}(c,D)=\sum\limits_{c\in H}\frac{3^5-3^3y(P)}{3^3-x(P)^3}.
\end{align*}

To determine $\mathrm{ord}_3(x(P))$ and $\mathrm{ord}_3(y(P))$, recall that the change of variables $x=u^2X+r$, $y=2u^3Y$ where \mbox{$r=-\frac{3}{2}\sqrt[3]{\frac{-13-3\sqrt{-3}}{2}}$} gives us a model of $E$ having good reduction at $3$ (see Proposition \ref{prop4} of Appendix A). In terms of $X$ and $Y$, we have
\begin{align*}\sum\limits_{c\in V}\mathscr{M}(c,D)&=\sum\limits_{c\in H}\frac{3^5-2\cdot 3^3u^3Y(P)}{3^3-r^3-u^6X(P)^3-3u^4rX(P)^2-3u^2r^2X(P)}.\\
\end{align*}

Now, $P$ is a torsion of point of $E$ of order prime to $3$ and $E$ has good reduction at $3$ so $\mathrm{ord}_3(X(P)), \mathrm{ord}_3(Y(P))\geqslant 0$.  If $\mathrm{ord}_3(Y(P))>0$, $P$ reduces to a $2$-torsion after reduction modulo $3$, but $P$ is a $D$-torsion and reduction modulo $3$ is injective, hence we must have $\mathrm{ord}_3(Y(P))=0$. Now, $\mathrm{ord}_3(3^3-r^3)=\mathrm{ord}_3(3^3(1-s^3))$, where $r=3s$.  Also,
\begin{align*}1-s^3&=1+\left(\frac{1}{2}\sqrt[3]{\frac{-13-3\sqrt{-3}}{2}}\right)^3\\
&=\frac{3-3\sqrt{-3}}{16},
\end{align*}
so $\mathrm{ord}_3(1-s^3)=1$. In addition, we have $\mathrm{ord}_3(u)=\frac{3}{4}$ and $\mathrm{ord}_3(r)=1$. Therefore, $\mathrm{ord}_3\left(u^6X(P)^3+3u^4rX(P)^2+3u^2r^2X(P)\right)>4=\mathrm{ord}_3(3^3-r^3)$. It follows that
\begin{align*}\mathrm{ord}_3\left(\sum\limits_{c\in V}\mathscr{M}(c,D)\right)&\geqslant\mathrm{ord}_3(3^5)-\mathrm{ord}_3(3^3-r^3)\\
&=1.
\end{align*}

On the other hand, by Proposition \ref{prop2.3.8}, we have $9\mid \#(V)$. Hence,
\begin{align*}
\mathrm{ord}_3\left(\sum\limits_{c\in V}\mathcal{E}_1^*\left(\frac{c\O}{D}+\frac{\Omega}{3}, \mathcal{L}\right)\right)&=\min\left(\mathrm{ord}_3\left(\frac{1}{2}\sum\limits_{c\in V}\mathscr{M}(c,D)\right), \mathrm{ord}_3(\#(V))\right)\\
&= 1
\end{align*}
as required.
\end{proof}

Recall from Remark \ref{rem1} that $\frac{L(\overline{\psi},1)}{\Omega}=\frac{1}{3}$. It follows from Theorem \ref{thm2.3.5} and Corollary \ref{cor2.3.7} that
\begin{thm}\label{thm2.3.9} Let be a cube-free product of cubic special primes, and let $n$ be the number of distinct prime factors of $D$ in $K$. Then
$$\mathrm{ord}_3(\Phi_{D^2})\geqslant n.$$
\end{thm}

We can generalise Definition \ref{def2} as follows.
\begin{defn} Given a character $\chi: \left(\mathbb{Z}/3\mathbb{Z}\right)^n\to \mathbb{C}^\times$, define
$$\Phi_{D^2}^{(\chi)}=\sum\limits_{\alpha\in\left(\mathbb{Z}/3\mathbb{Z}\right)^n}\chi(\alpha)\frac{L_{S_\alpha}(\overline{\psi}_{D_\alpha^2},1)}{\Omega}.$$
\end{defn}

Using essentially the same arguments that are used to prove Theorem \ref{thm2.3.9}, we can show:

\begin{lem}\label{lem2.3.11} For any character $\chi: \left(\mathbb{Z}/3\mathbb{Z}\right)^n\to \mathbb{C}^\times$, we have
$$\mathrm{ord}_3(\Phi_{D^2}^{(\chi)})\geqslant n.$$
\end{lem}

\begin{proof}
By equation \eqref{eq3.5}, we have
$$\chi(\alpha)\frac{L_S(\overline{\psi}_{D_\alpha^2}, 1)}{\Omega}=\frac{1}{3D}\sum\limits_{\mathfrak{b}\in\mathcal{B}}\chi(\alpha)\left(\frac{D_\alpha}{\mathfrak{b}}\right)_3 \mathcal{E}_1^*\left(\frac{\Omega}{3D}, \mathcal{L}\right)^{\sigma_\mathfrak{b}}.$$
Also, by the law of cubic reciprocity, we have
\begin{align*}
\left(\frac{D_\alpha}{3c+D_\alpha}\right)_3=\left(\frac{3c+D_\alpha}{D_\alpha}\right)_3=\left(\frac{3c}{D_\alpha}\right)_3=\left(\frac{c}{D_\alpha}\right)_3.
\end{align*}
Let $n=n_\alpha$. Then we have a $1$-dimensional character $\epsilon_3^{(\chi)}(\cdot, c) : (\mathbb{Z}/3\mathbb{Z})^n\to \boldsymbol{\mu}_3$ defined by $\epsilon_3^{(\chi)}(\alpha, c)=\chi(\alpha)\left(\frac{c}{D_\alpha}\right)_3$. Now, considering its inner product with the trivial character gives
$$
\sum\limits_{\alpha\in (\mathbb{Z}/3\mathbb{Z})^n}\epsilon_3^{(\chi)}(\alpha, c) = \left\{ \begin{array}{ll}
 3^n &\mbox{ if $c\in V^{(\chi)}$} \\
 0 &\mbox{ otherwise,}
       \end{array} \right.
$$
where  $V^{(\chi)}=\{c\in \mathcal{C} \;\; : \left(\frac{c}{D_\alpha}\right)_3=\chi(\alpha)^2 \text{ for all }\alpha\in(\mathbb{Z}/3\mathbb{Z})^n\}$.
Thus
$$\Phi_{D^2}^{(\chi)}=3^n\sum\limits_{c\in V^{(\chi)}} \frac{1}{3D}\mathcal{E}_1^*\left(\frac{c\O}{D}+\frac{\Omega}{3}, \mathcal{L}\right).$$
Recall that for any prime $\pi_j$ dividing $D_\alpha$, we have \mbox{$\left(\frac{\omega}{\pi_j}\right)_3=1$}. Hence
$$\left(\frac{c}{D_\alpha}\right)_3=\left(\frac{\omega c}{D_\alpha}\right)_3=\left(\frac{\omega^2 c}{D_\alpha}\right)_3,$$
so $c\in V^{(\chi)}$ implies $w c$, $\omega^2 c\in V^{(\chi)}$.  Also, the proof of Proposition \ref{prop2.3.8} shows that $V^{(\chi)}=V_{(\alpha_1,...,\alpha_n)}$ where $\alpha_i\in\{0,1,2\}$ is such that $\chi(e_i)=\omega^{\alpha_i}$, where $e_i\in (\mathbb{Z}/3\mathbb{Z})^n$ has $1$ in the $i$-th entry and $0$ elsewhere. Hence,  $\#(V)=\#(V^{(\chi)})$ or $\#(V^{(\chi)})=0$, so in either case we have $9\mid \#(V^{(\chi)})$. So we can apply the proofs of Theorem \ref{thm2.3.6} and Corollary \ref{cor2.3.7}, and obtain
$$\mathrm{ord}_3\left(\sum\limits_{c\in V^{(\chi)}}\mathcal{E}_1^*\left(\frac{c\O}{D}+\frac{\Omega}{3}, \mathcal{L}\right)\right)\geqslant 1,$$
so the result follows.

\end{proof}

\begin{remark} We note that the assumption $\mathrm{ord}_3(\pi-1)\geqslant 2$ for any prime factors $\pi$ of $D$ is essential. If we take $\pi=55+33\omega$ and $S=\{\pi\}$, then $\mathrm{ord}_3(\pi-1)=\frac{3}{2}$ and $\mathrm{N}(\pi)\equiv 1\bmod 27$. Then we have $\mathrm{ord}_3\left(\frac{L_{S}(\overline{\psi},1)}{\Omega}\right)=\frac{1}{2}$, but a computation shows $\frac{L(\overline{\psi}_{\pi^2},1)\sqrt[3]{\pi}}{\Omega}=3$ and $\frac{L(\overline{\psi}_{\pi^4},1)\sqrt[3]{\pi^2}}{\Omega}=289$, so that $\mathrm{ord}_3(\Phi_{\pi^2})=0$. Note also that we used $\pi\equiv 1\bmod 9$ when showing $\left(\frac{3}{\pi}\right)_3=1$, which is not true when $\mathrm{ord}_3(\pi-1)=\frac{3}{2}$.
\end{remark}

Since we required that  $\mathrm{ord}_3(\pi-1)\geqslant 3$ and that $9$ divides the order of $1-\omega$ in $\left(\mathcal{O}_K/\pi\mathcal{O}_K\right)^\times$ for any prime $\pi$ of $K$ dividing $D$, we can improve the bound in Lemma \ref{lem2.3.11} slightly by a similar proof. This can be found in Corollary \ref{cor6}, Appendix A, and we will only use this in the case $n=1$. We are ready to prove the second main result:

\begin{thm}\label{thm2.3.13} We have 
 $$\mathrm{ord}_3\left(\frac{L(\overline{\psi}_{D^2},1)}{\Omega}\right)\geqslant \frac{1}{2}(n+1).$$
\end{thm}

\begin{proof}
We prove this by induction on $n$. First, write  $\alpha=(\alpha_1,\ldots , \alpha_n)$ for the element in $(\mathbb{Z}/3\mathbb{Z})^n$ with $D=D_\alpha$. Given $\beta, \gamma\in (\mathbb{Z}/3\mathbb{Z})^{n}$, we write $\beta< \gamma$ if $D_\beta\mid D_\gamma$ but $D_\beta\neq D_\gamma$. Let $n_\alpha=1$ and $S_\alpha=\{\pi_1\}$, say. Then we consider
$$\Phi_{\pi_1^2}=\frac{L_{S_\alpha}(\overline{\psi},1)}{\Omega}+\frac{L_{S_\alpha}(\overline{\psi}_{\pi_1^2},1)}{\Omega}+\frac{L_{S_\alpha}(\overline{\psi}_{\pi_1^4},1)}{\Omega},$$
where the last two terms are primitive. Also,
\begin{align*}\frac{L_{S_\alpha}(\overline{\psi},1)}{\Omega}&=\left(1-\frac{\overline{\psi}((\pi_1))}{\pi_1\overline{\pi_1}}\right)\frac{L(\overline{\psi}, 1)}{\Omega}\\
&=\left(\frac{\pi_1-1}{\pi_1}\right)\frac{1}{3}
\end{align*}
and $\pi_1\equiv 1\bmod 9$. Hence $\mathrm{ord}_3\left(\frac{\pi_1-1}{\pi_1}\right)\geqslant 2$, and
$$\mathrm{ord}_3\left(\frac{L_{S_\alpha}(\overline{\psi},1)}{\Omega}\right)\geqslant 2-1=1.$$
Now let $\chi_1: \mathbb{Z}/3\mathbb{Z}\to \boldsymbol{\mu}_3$ be the character defined by $1\mapsto \omega$ and let  $\chi_2: \mathbb{Z}/3\mathbb{Z}\to \boldsymbol{\mu}_3$ be the character defined by $1\mapsto \omega^2$. Then we have
$$\Phi_{\pi_1^2}^{(\chi_i)}=\frac{L_{S_\alpha}(\overline{\psi},1)}{\Omega}+\omega^i\frac{L_{S_\alpha}(\overline{\psi}_{\pi_1^2},1)}{\Omega}+\omega^{2i} \frac{L_{S_\alpha}(\overline{\psi}_{\pi_1^4},1)}{\Omega},$$
for $i=1,2$.
Hence we obtain
$$\Phi_{\pi_1^2}-\omega\Phi_{\pi_1^2}^{(\chi_1)}=(1-\omega)\frac{L_{S_\alpha}(\overline{\psi},1)}{\Omega}+(1-\omega^2)\frac{L_{S_\alpha}(\overline{\psi}_{\pi_1^2},1)}{\Omega}.$$
We know that $\mathrm{ord}_3(\Phi_{\pi_1^2}-\omega\Phi_{\pi_1^2}^{(\chi_1)})\geqslant \frac{5}{4}$ (see  Corollary \ref{cor6} of Appendix A), and we also checked that $\mathrm{ord}_3\left(\frac{L_{S_\alpha}(\overline{\psi},1)}{\Omega}\right)\geqslant 1$, so $\mathrm{ord}_3\left((1-\omega)\left(\frac{L_{S_\alpha}(\overline{\psi},1)}{\Omega}\right)\right)\geqslant \frac{3}{2}$. It follows that
$$\mathrm{ord}_3\left((1-\omega^2)\left(\frac{L_{S_\alpha}(\overline{\psi}_{\pi_1^2},1)}{\Omega}\right)\right)\geqslant \frac{5}{4},$$
that is,
$$\mathrm{ord}_3\left(\frac{L_{S_\alpha}(\overline{\psi}_{\pi_1^2},1)}{\Omega}\right)\geqslant \frac{3}{4}.$$
But $\frac{L_{S_\alpha}(\overline{\psi}_{\pi_1^2},1)\sqrt[3]{\pi_1^2}}{\Omega}\in K$ so $\mathrm{ord}_3\left(\frac{L_{S_\alpha}(\overline{\psi}_{\pi_1^2},1)}{\Omega}\right)$ must be an integer multiple of $\frac{1}{2}$. Hence 
$$\mathrm{ord}_3\left(\frac{L_{S_\alpha}(\overline{\psi}_{\pi_1^2},1)}{\Omega}\right)\geqslant 1=\frac{1}{2}(n_\alpha+1)$$
as required.

Now suppose the result holds for all $n_\beta< n_\alpha$, where $\beta< \alpha$. We have
\begin{align*}\Phi_{D_\alpha^2}&=\frac{L_{S_\alpha}(\overline{\psi},1)}{\Omega}+\sum\limits_{n_\beta<n_\alpha}\frac{L_{S_\alpha}(\overline{\psi}_{
D_\beta^2},1)}{\Omega}+\sum\limits_{n_\gamma=n_\alpha}\frac{L_{S_\alpha}(\overline{\psi}_{D_\gamma^2},1)}{\Omega}
\end{align*}
where the terms in the last summand are primitive. 

We know that 
\begin{align*}\frac{L_{S_\alpha}(\overline{\psi},1)}{\Omega}&=\prod\limits_{\pi\in S_\alpha}\left(1-\frac{\overline{\psi}((\pi))}{\pi\overline{\pi}}\right)\frac{L(\overline{\psi},1)}{\Omega}\\
&=\prod\limits_{\pi\in S_\alpha}\left(\frac{\pi-1}{\pi}\right)\frac{1}{3}
\end{align*}
and $\pi\equiv 1\bmod 27$, so $\mathrm{ord}_3\left(\frac{L_{S_\alpha}(\overline{\psi},1)}{\Omega}\right)\geqslant 3n_\alpha-1$. Next, for $n_\beta<n_\alpha$, we have
$$\frac{L_{S_\alpha}(\overline{\psi}_{D_\beta^2},1)}{\Omega}=\prod\limits_{\pi\in S_\alpha\backslash S_\beta}\left(1-\frac{\overline{\psi}_{D_\beta^2}((\pi))}{\pi\overline{\pi}}\right)\frac{L(\overline{\psi}_{D_\beta^2},1)}{\Omega}$$
and $\psi_{D_\beta^2}((\pi))=\left(\frac{D_\beta}{\pi}\right)_3\pi=\omega^i \pi$, $i\in\{0,1,2\}$. Furthermore, by the induction hypothesis, \mbox{$\mathrm{ord}_3\left(\frac{L(\overline{\psi}_{D_\beta^2},1)}{\Omega}\right)\geqslant \frac{1}{2}(n_\beta+1)$}. It follows that
\begin{align*}\mathrm{ord}_3\left(\sum\limits_{n_\beta<n_\alpha}\frac{L_{S_\alpha}(\overline{\psi}_{D_\beta^2},1)}{\Omega}\right)&\geqslant \frac{1}{2}(n_\alpha-n_\beta)+\frac{1}{2}(n_\beta+1)\\
&=\frac{1}{2}(n_\alpha+1).\end{align*}

We also know by Lemma \ref{lem2.3.11} that $\mathrm{ord}_3(\Phi_{D_\alpha^2}^{(\chi)})\geqslant n_\alpha$ for any character $\chi: (\mathbb{Z}/3\mathbb{Z})^n\to \boldsymbol{\mu}_3$.

To find $\mathrm{ord}_3\left(\frac{L(\overline{\psi}_{D_\gamma^2},1)}{\Omega}\right)$ for $\gamma=(\gamma_1,\ldots , \gamma_n)\in\left(\mathbb{Z}/3\mathbb{Z}\right)^n$ with $n_\gamma=n_\alpha$, suppose first that $\gamma\neq (2,\ldots, 2)$, so there exists $j\in\{1,\ldots n\}$ with $\gamma_j=1$. Without loss of generality, we may assume $j=1$. Let $\chi_1: \left(\mathbb{Z}/3\mathbb{Z}\right)^n=\langle g_1,\ldots , g_n\rangle\to \boldsymbol{\mu}_3$ be the character defined by $\chi_1(g_1)=\omega$ and $\chi_1(g_j)=1$ for $j=2,\ldots , n$, and let  $\chi_2: \left(\mathbb{Z}/3\mathbb{Z}\right)^n=\langle g_1,\ldots , g_n\rangle\to \boldsymbol{\mu}_3$ be the character defined by $\chi_2(g_1)=\omega^2$ and $\chi_2(g_j)=1$ for $j=2,\ldots , n$. Then, by writing out $\Phi_{D_\alpha}-\omega\Phi_{D_\alpha}^{(\chi_1)}$ explicitly, we see that $\mathrm{ord}_3\left(\sum\limits_{\substack{n_\gamma<n_\alpha \\ \gamma_j=1}}\frac{L(\overline{\psi}_{D_\gamma^2},1)}{\Omega}\right)\geqslant \frac{1}{2}(n_\alpha+1)$ for any $j=1,\ldots , n$, and similarly  $\mathrm{ord}_3\left(\sum\limits_{\substack{n_\gamma<n_\alpha \\ \gamma_j=2}}\frac{L(\overline{\psi}_{D_\gamma^2},1)}{\Omega}\right)\geqslant \frac{1}{2}(n_\alpha+1)$ for any $j=1,\ldots , n$.

Now let $\chi_2$ be the character defined by $g_1\mapsto \omega$, $g_2\mapsto \omega$ and $g_j\mapsto 1$ for $j\neq 1,2$, and let $\chi_3$ be the character defined by $g_1\mapsto \omega^2$, $g_2\mapsto \omega$ and $g_j\mapsto 1$ for $j\neq 1,2$. Then an easy calculation gives
\begin{align*}(\Phi_{D_\alpha}^{(\chi_2)}-\omega\Phi_{D_\alpha}^{(\chi_3)})-(\Phi_{D_\alpha}-\omega\Phi_{D_\alpha}^{(\chi_1)})&=3\omega\sum_{\substack{n_\beta<n_\alpha \\ \beta_1=0, \beta_2=1}}\frac{L_{S_\alpha}(\overline{\psi}_{D_\beta^2},1)}{\Omega}-3\sum_{\substack{n_\beta<n_\alpha \\ \beta_1=0, \beta_2=2}}\frac{L_{S_\alpha}(\overline{\psi}_{D_\beta^2},1)}{\Omega}\\
&-3\sum_{\substack{n_\beta<n_\alpha \\ \beta_1=1, \beta_2=0}}\frac{L_{S_\alpha}(\overline{\psi}_{D_\beta^2},1)}{\Omega}+3\omega^2\sum_{\substack{n_\beta<n_\alpha \\ \beta_1=1, \beta_2=1}}\frac{L_{S_\alpha}(\overline{\psi}_{D_\beta^2},1)}{\Omega}\\
&+3\omega^2\sum_{\substack{n_\gamma=n_\alpha \\ \gamma_1=1, \gamma_2=1}}\frac{L(\overline{\psi}_{D_\gamma^2},1)}{\Omega}.
\end{align*}
So we have
$$\mathrm{ord}_3\left(\sum_{\substack{n_\gamma=n_\alpha \\ \gamma_1=1, \gamma_2=1}}\frac{L(\overline{\psi}_{D_\gamma^2},1)}{\Omega}\right)\geqslant \frac{1}{2}(n_\alpha+1).$$
Similarly, we can show
$$\mathrm{ord}_3\left(\sum_{\substack{n_\gamma=n_\alpha \\ \gamma_i=e_i, \gamma_j=e_j}}\frac{L(\overline{\psi}_{D_\gamma^2},1)}{\Omega}\right)\geqslant \frac{1}{2}(n_\alpha+1)$$
for any $e_i, e_j\in\{1,2\}$ with $i\neq j$. Now we claim the following:

\begin{lem}
Let $\gamma\in (\mathbb{Z}/3\mathbb{Z})^n$ be such that $n_\gamma=n_\alpha$. Then for any $J\subset \{1,\ldots , n\}$ and any $e_j\in \{1,2\}$ for $j\in J$ , we have $$\mathrm{ord}_3\left(\sum\limits_{\substack{\gamma_j=e_j \\j\in J}}\frac{L(\overline{\psi}_{D_\gamma^2},1)}{\Omega}\right)\geqslant M,$$
where $M\in \mathbb{Q}$ is such that $\mathrm{ord}_3\left(\sum\limits_{\substack{\gamma\in (\mathbb{Z}/3\mathbb{Z})^n\\ n_\gamma=n_\alpha}}\frac{L(\overline{\psi}_{D_\gamma^2},1)}{\Omega}\right)\geqslant M$.
\end{lem}
 \begin{proof} We prove this by induction on $|J|$. The cases $|J|=1,2$ were established above. Given $J\subset \{1,\ldots ,n\}$ and $e_j\in\{1,2\}$ for $j\in J$, let $X_J$ denote the sum
$$X_J:=\sum\limits_{\substack{\gamma_j=e_j \\j\in J}}\frac{L(\overline{\psi}_{D_\gamma^2},1)}{\Omega}.$$
Now suppose the lemma is true for any $J\subset \{1,\ldots , n\}$ with $|J|=k>1$. Then let $|J|=k+1$, and without loss of generality, we may assume $J=\{1,\ldots , k+1\}$. Pick $e_{j}\in\{1,2\}$ for $j\in J$. Then by the induction hypothesis, 
$\mathrm{ord}_3(X_{\{1,\ldots , k\}})\geqslant M$ and $\mathrm{ord}_3(X_{\{2,\ldots , k+1\}})\geqslant M$. Now, 
\begin{align*}X_{\{1,\ldots , k\}}-X_{\{2,\ldots k+1\}}&=\sum\limits_{\substack{\gamma_j=\epsilon_j j\in \{2,\ldots k\}\\ \gamma_1=e_1, \gamma_{k+1}\neq e_{k+1}}}\frac{L(\overline{\psi}_{D_\gamma^2},1)}{\Omega}-\sum\limits_{\substack{\gamma_j=\epsilon_j j\in \{2,\ldots k\} \\ \gamma_1\neq e_1, \gamma_{k+1}= e_{k+1}}}\frac{L(\overline{\psi}_{D_\gamma^2},1)}{\Omega}\\
&=A-B,
\end{align*}
say. Now, $A+B+X_J=X_{\{2,\ldots , k\}}$ so $\mathrm{ord}_3(A+B+X_J)\geqslant M$. On the other hand,
\mbox{$X_{\{1,\ldots k\}}+X_{\{2,\ldots k+1\}}=A+B+2X_J$} so $\mathrm{ord}_3(A+B+2X_J)\geqslant M$.  It follows that $\mathrm{ord}_3(X_J)\geqslant M$ as required.
\end{proof}

Hence applying the above lemma with $J=\{1, \ldots, n\}$, we see that for any $\gamma\in (\mathbb{Z}/3\mathbb{Z})^n$ and $n_\gamma=n_\alpha$, we have
 $$\mathrm{ord}_3\left(\frac{L(\overline{\psi}_{D_\gamma^2},1)}{\Omega}\right)\geqslant \frac{1}{2}(n_\alpha+1)$$
and the result follows.
\end{proof}

The following is an immediate consequence of Theorem \ref{thm2.3.13}.
\begin{thm}\label{thm2.3.15} 
Let $D>1$ be an integer which is a cube-free product of cubic-special primes. Then 
$$\mathrm{ord}_3\left(L^{(\text{alg})}\left(E(D^2),1\right)\right)\geqslant k(D)+1,$$
where $k(D)$ is the number of distinct rational prime factors of $D$.
\end{thm}

\begin{proof} The number of distinct primes in $K$ dividing $D$ is twice the number of distinct rational primes dividing $D$, so by Theorem \ref{thm2.3.13},
$$\mathrm{ord}_3\left(L^{(\text{alg})}\left(\overline{\psi}_{D^2},1\right)\right)\geqslant \frac{1}{2}(2(k(D)+1))=k(D)+\frac{1}{2}.$$
But we know $L^{(\text{alg})}\left(\overline{\psi}_{D^2},1\right)\in\mathbb{Q}$, so $\mathrm{ord}_3\left(L^{(\text{alg})}\left(\overline{\psi}_{D^2},1\right)\right)\geqslant k(D)+1$ as required.
\end{proof}

\begin{remark}\label{rem3}The bound in Theorem \ref{thm2.3.15} is sharp. For example, let $\pi=28+27\omega$ and let $D=\mathrm{N}(\pi)=757$, which is a rational prime. Then we have $L^{(\text{alg})}(E({D^2}),1)=9$ so $\mathrm{ord}_3\left(L^{(\text{alg})}(E({D^2}),1)\right)=2$.
\end{remark}
In fact, the numerical examples listed in Appendix B suggest that Theorem \ref{thm2.3.15} is true whenever $D>1$ is an odd integer congruent to $1$ modulo $9$ whose prime factors are congruent to $1$ modulo $3$. Finally, we note that the condition $D\equiv 1\bmod 9$ is not sufficient. Indeed, for $D=55$ we have $L^{(\text{alg})}(E({D^2}),1)=3$.

\appendix 

\section{}

\begin{lem}\label{lem4}
A rational prime $p$ is a special split prime if and only if it splits in $K$, and  $\psi(\mathfrak{p})\equiv \pm 1\bmod 4$ for both of the primes $\mathfrak{p}$ of $K$ above $p$. Moreover, $K(x(E[4]))=K(\boldsymbol{\mu}_{4}, \sqrt[3]{2})$.
\end{lem}

\begin{proof}
Put $F=K(E[4])$, and let $G$ denote the Galois group of $F$ over $K$. Since $E$ has good reduction at $2$, the action of $G$ on $E[4]$ defines an isomorphism
$$j: G\xrightarrow{\sim} \mathrm{Aut}_{\mathcal{O}_K}\left(E[4]\right)=\left(\mathcal{O}_K/4\mathcal{O}_K\right)^\times.$$
In particular, it follows that $[F:K]=12$, since $2$ is inert in $K$. Let $\tau$ denote the unique element of $G$ such that $j(\tau)= -1\bmod 4\mathcal{O}_K$. Then the field $L=K(x(E[4]))$ is the fixed field of $\tau$, so that $[L:K]=6$. Clearly, $K(E[2])=K(\sqrt[3]{2})$. Also by Weil pairing, we have $\boldsymbol{\mu}_4\subset F$. We claim that $L=K(\boldsymbol{\mu}_{4}, \sqrt[3]{2})$. We know that $E[2]=\{\mathcal{O} , (\frac{\sqrt[3]{2}\cdot 3}{2},0),  (\frac{\sqrt[3]{2}\cdot 3}{2}\omega,0),  (\frac{\sqrt[3]{2}\cdot 3}{2}\omega^2,0)\}$. Using the doubling formula, we get that the $x$-coordinate of a point in $E[4]\backslash E[2]$ satisfies
$$\frac{x^4+2\cdot 3^3x}{4x^3-3^3}=\frac{\sqrt[3]{2}\cdot 3}{2}.$$
Let $x=\frac{\sqrt[3]{2}\cdot 3}{2}z$, then the equation becomes
$$z^4-4z^3+8z+4=(z^2-2z-2)^2=0,$$
which has roots $z=1\pm \sqrt{3}$ each with multiplicity $2$. Hence the $x$-coordinate of a point in $E[4]\backslash E[2]$ is
$x=\frac{\sqrt[3]{2}\cdot 3 (1\pm\sqrt{3})}{2}\in K(\boldsymbol{\mu}_{4}, \sqrt[3]{2})$, as required.
Now let $p$ be any prime which splits in $K$, and let $\mathfrak{p}$ be one of the prime ideals of $K$ above $p$. Then the Frobenius automorphism of $K$ acts on $E[4]$ by multiplication by $\psi(\mathfrak{p})$, thanks to the main theorem of complex multiplication. It follows that $\mathfrak{p}$ splits completely in $F$ if and only if $\psi(\mathfrak{p})\equiv 1\bmod 4$, and $\mathfrak{p}$ splits completely in $L$ if and only if $\psi(\mathfrak{p})\equiv \pm 1\bmod 4$. 
\end{proof}

\begin{prop}\label{prop4}
Over the field
\begin{align*}F=K\left(\sqrt[6]{\frac{27+3\sqrt{-3}}{2}}\right),
\end{align*}
there exists a change of variables $x=u^2X+r$, $y=2u^3Y$ with $u, r\in F$ which gives the following equation for $E$
$$Y^2=X^3+\frac{(9+\sqrt{-3})}{4}X^2+\frac{13+3\sqrt{-3}}{8}X+\frac{2+\sqrt{-3}}{8}$$
which has good reduction at $3$. Here, $u=\frac{\sqrt{\alpha}}{\beta^2}$ where $\alpha=\frac{27+3\sqrt{-3}}{2}$, $\beta=\sqrt[3]{\frac{1-3\sqrt{-3}}{2}}$ and \mbox{$r=-\frac{3}{2}\sqrt[3]{\frac{-13-3\sqrt{-3}}{2}}$}.
\end{prop}

\begin{proof} Note that for our curve, the smallest split prime is $7$. So one should try to find an explicit equation for the curve $E$ over the field $F= K(E[2+\sqrt{-3}])$ having good reduction at $3$ (see \cite[Theorem 2]{coa-wil}). The conductor of $F$ over $K$ is $(3(2+\sqrt{-3}))$, since the conductor of the Gr\"{o}ssencharacter of $E/K$ is $3\mathcal{O}_K$. Furthermore, $F/K$ is an abelian extension of degree $6$ and the group $\mu_6\subset K$. Thus, by Kummer theory, we must
have $F = K(\sqrt[6]{\alpha})$, for some $\alpha\in K^*$. The only primes of $K$ which can ramify in $F$ are those dividing $7$, $3$ and $w$, so the Kummer generator $\alpha$ must be of the form $(2+\sqrt{-3})^a\cdot (\omega-1)^b\cdot (-\omega)^c$ where $a,b,c\in \{0,\ldots , 5\}$. Recall from the theory of complex multiplication that for a prime ideal $\mathfrak{p}$ of $K$ prime to $3$, we have $\psi_{E/K}(\mathfrak{p})=\pi$ where $\pi$ is the unique generator of $\mathfrak{p}$ which is $1\bmod 3\mathcal{O}_K$. Now, suppose in addition that $\mathfrak{p}$ is prime to $7$. Then $F/K$ is unramified at $\mathfrak{p}$ so
$$\mathrm{Frob}_{\mathfrak{p}}=\psi_{E/K}(\mathfrak{p}).$$
If we pick a prime $\mathfrak{p}=(\pi)$ such that $\pi\equiv 1\bmod 3\mathcal{O}_K$ and $\pi\equiv 1\bmod (2+\sqrt{-3})\mathcal{O}_K$, then we have
$$(P)^{\mathrm{Frob}_\mathfrak{p}}=\psi_{E/K}(\mathfrak{p})(P)=\pi(P)=P$$
for $P\in E[2+\sqrt{-3}]$, since $\pi\equiv 1\bmod (2+\sqrt{-3})\mathcal{O}_K$. So $\psi_{E/K}(\mathfrak{p})$ is the identity in the extension $K(E[2+\sqrt{-3}])/K$. On the other hand, $K(E[2+\sqrt{-3}])=K(\sqrt[6]{\alpha})$ and we know that
$$(\sqrt[6]{\alpha})^{\mathrm{Frob}_{\mathfrak{p}}}\equiv (\sqrt[6]{\alpha})^{\mathrm{N}(\mathfrak{p})}\bmod \mathfrak{p},$$
so for $\mathrm{Frob}_{\mathfrak{p}}$ to be the identity, it is necessary that
$$(\sqrt[6]{\alpha})^{\mathrm{N}(\mathfrak{p})}\equiv \sqrt[6]{\alpha}\bmod \mathfrak{p}.$$
We eliminate the possibilities for $(a, b, c)$ by trying out some examples.

\begin{eg}\label{ex1} Let $\pi=13+6\sqrt{-3}$ and $\mathfrak{p}=(\pi)$. Then $\pi\equiv 1\bmod 3\mathcal{O}_K$, $\pi\equiv 1\bmod (2+\sqrt{-3})\mathcal{O}_K$ and $\mathrm{N}(\mathfrak{p})=277$. So $(\sqrt[6]{\alpha})^{\mathrm{Frob}_\mathfrak{p}}\equiv (\sqrt[6]{\alpha})^{277}\equiv (\sqrt[6]{\alpha})\alpha^{46}$. Thus, for $\mathrm{Frob}_\mathfrak{p}$ to be the identity, we need
$$\alpha^{46}\equiv \left(2+\sqrt{-3}\right)^{46a}\left(\frac{-3+\sqrt{-3}}{2}\right)^{46b}\left(\frac{1-\sqrt{-3}}{2}\right)^{46c}\equiv 1\bmod \mathfrak{p}.$$
But $13+6\sqrt{-3}\equiv 0\bmod \mathfrak{p}$ so we can replace $\sqrt{-3}$ with $\frac{-13}{6}$ and now that we have rational numbers, we can replace $\bmod \;\mathfrak{p}$ with $\bmod\; \mathrm{N}(\mathfrak{p})$. Hence the equation becomes
$$\left(2-\frac{13}{6}\right)^{46a}\left(\frac{-3-\frac{13}{6}}{2}\right)^{46b}\left(\frac{1+\frac{13}{6}}{2}\right)^{46c}\equiv 1\bmod 277.$$ 
Also, $6^{-1}\equiv -46\bmod 277$ and $2^{-1}\equiv 139\bmod 277$, so 
$$\left(2+46\cdot 13\right)^{46a}\left(139(-3+46\cdot 13\right)^{46b}\left(139(1-46\cdot 13)\right)^{46c}\equiv 1\bmod 277,$$
that is,
\begin{equation}\label{eqex1}117^a\cdot 276^b\cdot 160^c\equiv 1\bmod 277.
\end{equation}
\end{eg}

\begin{eg}\label{ex2} Let $\pi=1+\frac{1+\sqrt{-3}}{2}\cdot 3(2+\sqrt{-3})=\frac{5+9\sqrt{-3}}{2}$. Then $\pi\equiv 1\bmod 3\mathcal{O}_K$, \mbox{$\pi\equiv 1\bmod (2+\sqrt{-3})\mathcal{O}_K$} and $\mathrm{N}(\mathfrak{p})=67$.  So $(\sqrt[6]{\alpha})^{\mathrm{Frob}_\mathfrak{p}}\equiv (\sqrt[6]{\alpha})^{67}\equiv (\sqrt[6]{\alpha})\alpha^{11}$. Hence for $\mathrm{Frob}_\mathfrak{p}$ to be the identity, we need
$$\alpha^{11}\equiv \left(2+\sqrt{-3}\right)^{11a}\left(\frac{-3+\sqrt{-3}}{2}\right)^{11b}\left(\frac{1-\sqrt{-3}}{2}\right)^{11c}\equiv 1\bmod \mathfrak{p}.$$
But we now have $\sqrt{-3}\equiv\frac{5}{9}\bmod \mathfrak{p}$, $9^{-1}\equiv 15\bmod 67$ and $2^{-1}\equiv 34\bmod 67$ so the equation becomes
$$(2+15\cdot 5)^{11a}(34(-3+15\cdot 5))^{11b}(34(1-15\cdot 5))^{11c}\equiv 1\bmod 67$$
that is,
\begin{equation}\label{eqex2}29^a\cdot 37^b\cdot 38^c\equiv 1\bmod 67.
\end{equation}
\end{eg}

Comparing the solutions to \eqref{eqex1} and \eqref{eqex2} in Examples \ref{ex1} and \ref{ex2}, we find that the common solutions are $(a,b,c)=(0,0,0), (1,3,2), (2,0,4), (3,3,0), (4,0,2)$ and $(5,3,4)$. However, we know that $F/K$ is a degree $6$ extension, so the only possibilities are $(a,b,c)=(1,3,2)$ and $(5,3,4)$. But $\frac{(2+\sqrt{-3})(\omega-1)(-\omega)}{\sqrt[6]{(2+\sqrt{-3})^5(\omega-1)^3(-\omega)^4}}=\sqrt[6]{(2+\sqrt{-3})(\omega-1)^3(-\omega)^2}$, so the corresponding fields are the same. Hence
\begin{align*}K(E[2+\sqrt{-3}])&=K\left(\sqrt[6]{(2+\sqrt{-3})(\omega-1)^3(-\omega)^2}\right)\\
&=K\left(\sqrt[6]{\frac{27+3\sqrt{-3}}{2}}\right).
\end{align*}
Let $E: y^2=4x^3-3^3$ and $x=u^2X+r$, $y=u^3Y$. Then in terms of $X,Y$, we have
$$u^6Y^2=4u^6X^3+12u^4rX^2+12u^2r^2X+4r^3-3^3$$
and $\mathrm{ord}_3\left(\sqrt[6]{\frac{27+3\sqrt{-3}}{2}}\right)=\frac{1}{4}$, so $\mathrm{ord}_3\left(\sqrt{\frac{27+3\sqrt{-3}}{2}}\right)=\frac{3}{4}$. We also have $\sqrt[3]{\alpha}=\sqrt[3]{2+\sqrt{-3}}\cdot \frac{-3+\sqrt{-3}}{2}\cdot \sqrt[3]{\left(\frac{1-\sqrt{-3}}{2}\right)^2}\in F$, so $\b=\sqrt[3]{(2+\sqrt{-3})\cdot\left(\frac{1-\sqrt{-3}}{2}\right)^2}=\sqrt[3]{\frac{1-3\sqrt{-3}}{2}}\in F$, so let $u=\frac{\sqrt{\alpha}}{\beta^2}$. Then $\mathrm{ord}_3(u)=\frac{3}{4}$ and $\mathrm{ord}_7(u)=-\frac{1}{12}$. If we divide the equation through by $u^6$, one can easily check that the discriminant of this curve is $u^{-12}\mathrm{disc}(E)$, so it is a $3$-adic unit and is integral at $7$. To make sure the coefficients of
$$Y^2=4X^3+\frac{12r}{u^2}X^2+\frac{12r^2}{u^4}X+\frac{4r^3-3^3}{u^6}$$
are integral at $3$, it is sufficient that $\mathrm{ord}_3(4r^2-3^3)\geqslant \mathrm{ord}_3(u^6)=\frac{9}{2}$. So we need $r=3s$ for some $s\in F$ and $\mathrm{ord}(4r^2-3^3)=\mathrm{ord}_3(3^3(4s^3-1))\geqslant \frac{9}{2}$, so $\mathrm{ord}_3(4s^3-1)\geqslant \frac{3}{2}$. 
Now, let 
\begin{align*}s=-\frac{\beta^2}{2}&=-\frac{1}{2}\sqrt[3]{(2+\sqrt{-3})^2\left(\frac{1-\sqrt{-3}}{2}\right)^4}\\
&=-\frac{1}{2}\sqrt[3]{\frac{-13-3\sqrt{-3}}{2}}.
\end{align*}
Then
\begin{align*}4s^3-1=\frac{13+3\sqrt{-3}}{4}-1=\frac{9+3\sqrt{-3}}{4}
\end{align*}
so $\mathrm{ord}_3(4s^3-1)=\frac{3}{2}$, as required.
Now, $r=3s=-\frac{3\beta^2}{2}$ and $u=\frac{\sqrt{\alpha}}{\beta^2}$, so
$$Y^2=4X^3-\frac{18\beta^6}{\alpha}X^2+\frac{27\beta^{12}}{\alpha^2}-\frac{27\beta^{12}(\beta^6+2)}{2\alpha^3}.$$

So substituting the values for $\alpha$ and $\beta$, we obtain an equation with coefficients in $K$:
$$Y^2=4X^3+(9+\sqrt{-3})X^2+\frac{13+3\sqrt{-3}}{2}X+\frac{2+\sqrt{-3}}{2}.$$
\end{proof}

\begin{cor}\label{cor6}
 For any character $\chi: \left(\mathbb{Z}/3\mathbb{Z}\right)^n\to \mathbb{C}^\times$, we have
$$\mathrm{ord}_3(\Phi_{D^2}^{(\chi)})\geqslant n+\frac{1}{4}.$$
\end{cor}
\begin{proof} We will assume for simplicity that $D$ is a prime power since we only use this Corollary in the case $n=1$. The proof for the case $n\geqslant 1$ is similar. Pick $\beta\in\mathcal{O}_K$ be such that $(1-\omega)\beta\equiv 1\bmod D$.
Let $\mathcal{C}$ be a set of elements of $\mathcal{O}_K$ such that $c \bmod D$ runs over $\left(\mathcal{O}_K/D\mathcal{O}_K\right)^\times$ precisely once and $\mathcal{C}$ can be written as a union of sets $\mathcal{C}=\bigcup\limits_{i\in\{0,1,2\}}\omega^i \mathcal{H}\bigcup\limits_{i\in\{0,1,2\}} \omega^i(1-\omega)\mathcal{H}\bigcup\limits_{i\in\{0,1,2\}}\omega^i\beta \mathcal{H}$ for some set $\mathcal{H}$. This is possible since $3$ and $D$ are coprime and $9$ divides the order of $1-\omega$ in $\left(\mathcal{O}_K/D\mathcal{O}_K\right)^\times$ by assumption. We will follow the notation in the proof of Lemma \ref{lem2.3.11}. Given $c\in V^{(\chi)}$, let $P$ be the point on $E: y^2=4x^3-3^3$ given by
$x(P)=\wp\left(\frac{c\O}{D},\mathcal{L}\right), y(P)=\wp'\left(\frac{c\O}{D},\mathcal{L}\right)$. Similarly let Q and R be the points given by $\left(x(Q),y(Q)\right)=\left(\wp\left(\frac{(1-\omega)c\O}{D},\mathcal{L}\right), \wp'\left(\frac{(1-\omega)c\O}{D},\mathcal{L}\right)\right)$ and $\left(x(R),y(R)\right)=\left(\wp\left(\frac{\beta c\O}{D},\mathcal{L}\right), \wp'\left(\frac{\beta c\O}{D},\mathcal{L}\right)\right)$ respectively, and define 
$$\mathscr{M}(c,D)=\frac{9-y(P)}{3-x(P)}.$$

We can write $V^{(\chi)}$ as a union of sets 
$$V^{(\chi)}=\bigcup\limits_{i\in\{0,1,2\}}\omega^iH\bigcup\limits_{i\in\{0,1,2\}} \omega^i(1-\omega)H\bigcup\limits_{i\in\{0,1,2\}}\omega^i\beta H$$
for some set $H$, since $\left(\frac{1-\omega}{D}\right)_3=\left(\frac{\beta}{D}\right)_3=1$. We wish to find $\mathrm{ord}_3\left(\sum\limits_{c\in V^{(\chi)}}\mathscr{M}(c,D)\right)$.
Recall that  $E$ has complex multiplication by $\omega$ via $\omega(x,y)=(\omega x,y)$, so $\wp'\left(\frac{\omega^i c\O}{D},\mathcal{L}\right)=\wp'\left(\frac{c\O}{D},\mathcal{L}\right)$. Moreover, $\mathcal{L}=\omega \mathcal{L}$ so $\wp\left(\frac{\omega^ic\O}{D},\mathcal{L}\right)=\wp\left(\frac{\omega^ic\O}{D},\omega^i\L\right)$ for $i=0, 1, 2$ and $\wp$ is homogeneous of degree $-2$ so
\begin{align*}\sum\limits_{i\in\{0,1,2\}}\frac{9-\wp'\left(\frac{w^ic\O}{D},\mathcal{L}\right)}{3-\wp\left(\frac{\omega^ic\O}{D},\mathcal{L}\right)}&=\frac{9-\wp'\left(\frac{c\O}{D},\mathcal{L}\right)}{3-\wp\left(\frac{c\O}{D},\mathcal{L}\right)}+\frac{9-\wp'\left(\frac{c\O}{D},\mathcal{L}\right)}{3-\omega\wp\left(\frac{c\O}{D},\mathcal{L}\right)}+\frac{9-\wp'\left(\frac{c\O}{D},\mathcal{L}\right)}{3-\omega^2\wp\left(\frac{c\O}{D},\mathcal{L}\right)}\\
&=\frac{3^5-3^3y(P)}{27-x(P)^3}.
\end{align*}

Furthermore, using the addition formula
$$\wp(z_1+z_2,\mathcal{L})=-\wp(z_1,\mathcal{L})-\wp(z_2,\mathcal{L})+\frac{1}{4}\left(\frac{\wp'(z_1,\mathcal{L})-\wp'(z_2,\mathcal{L})}{\wp(z_1,\mathcal{L})-\wp(z_2,\mathcal{L})}\right)^2,$$
and noting $\wp(z,\mathcal{L})$ is even and $\wp'(z,\mathcal{L})$ is odd, we get 
\begin{align*}\wp\left(\frac{(1-\omega)c\O}{D},\mathcal{L}\right)&=-\wp\left(\frac{c\O}{D},\mathcal{L}\right)-\wp\left(\frac{-\omega c\O}{D},\mathcal{L}\right)+\frac{1}{4}\left(\frac{\wp'(\frac{c\O}{D},\mathcal{L})-\wp'(\frac{-\omega c\O}{D},\mathcal{L})}{\wp\left(\frac{c\O}{D},\mathcal{L}\right)-\wp\left(\frac{-\omega c\O}{D},\mathcal{L}\right)}\right)^2\\
&=-(1+\omega)x(P)+\left(\frac{y(P)}{(1-\omega)x(P)}\right)^2.
\end{align*} 

Also, $\beta-\omega\beta\equiv 1\bmod D_a$ so 
\begin{align*}\wp\left(\frac{c\O}{D},\mathcal{L}\right)&=-\wp\left(\frac{\beta c\O}{D},\mathcal{L}\right)-\wp\left(\frac{-\omega \beta c\O}{D},\mathcal{L}\right)+\frac{1}{4}\left(\frac{\wp'(\frac{\beta c\O}{D},\mathcal{L})-\wp'(\frac{-\omega \beta c\O}{D},\mathcal{L})}{\wp(\frac{\beta c\O}{D},\mathcal{L})-\wp(\frac{-\omega\beta c\O}{D},\mathcal{L})}\right)^2\\
&=-(1+\omega)x(R)+\left(\frac{y(R)}{(1-\omega)x(R)}\right)^2.
\end{align*} 
Therefore,
\begin{align*}\sum\limits_{c\in V^{(\chi)}}\mathscr{M}(c,D)&=\sum\limits_{c\in H}\sum\limits_{i\in\{0,1,2\}}\frac{9-\wp'\left(\frac{w^ic\O}{D},\mathcal{L}\right)}{3-\wp\left(\frac{\omega^ic\O}{D},\mathcal{L}\right)}+\frac{9-\wp'\left(\frac{\omega^i(1-\omega)c\O}{D},\mathcal{L}\right)}{3-\wp\left(\frac{\omega^i(1-\omega)c\O}{D},\mathcal{L}\right)}+\frac{9-\wp'\left(\frac{w^i\beta c\O}{D},\mathcal{L}\right)}{3-\wp\left(\frac{\omega^i\beta c\O}{D},\mathcal{L}\right)},
\end{align*}
and this is equal to
\begin{equation}\label{eq3.1}\sum\limits_{c\in H}\frac{3^5-3^3y(P)}{27-x(P)^3}+\frac{3^5-3^3y(Q)}{27-\left(\omega^2x(P)+\left(\frac{y(P)}{(1-\omega)x(P)}\right)^2\right)^3}+\frac{3^5-3^3y(R)}{27-\left(\omega x(P)-\omega\left(\frac{y(R)}{(1-\omega)x(R)}\right)^2\right)^3}.
\end{equation}

To determine $\mathrm{ord}_3\left(\frac{y(P)}{(1-\omega)x(P)}\right)$, recall from Proposition \ref{prop4} that the change of variables $x=u^2X+r$, $y=2u^3Y$ where \mbox{$r=-\frac{3}{2}\sqrt[3]{\frac{-13-3\sqrt{-3}}{2}}$} gives us a model of $E$ having good reduction at $3$. In terms of $X$ and $Y$, we have
\begin{align*}\frac{y(P)}{(1-\omega)x(P)}=\frac{u^3Y(P)}{(1-\omega)(u^2X(P)+r)}.
\end{align*}
Now, $P$ is a torsion of point of $E$ of order prime to $3$ and $E$ has good reduction at $3$ so $\mathrm{ord}_3(X(P)), \mathrm{ord}_3(Y(P))\geqslant 0$. Also $\mathrm{ord}_3(u)=\frac{3}{4}$, and $\mathrm{ord}_3(r)=1$ so
$$\mathrm{ord}_3\left(\frac{y(P)}{(1-\omega)x(P)}\right)=\frac{3}{4}+\mathrm{ord}_3(Y(P)).$$
If  $\mathrm{ord}_3(Y(P))>0$, $P$ reduces to a $2$-torsion after reduction modulo $3$, but $P$ is a $D$-torsion and reduction modulo $3$ is injective, hence we must have $\mathrm{ord}_3(Y(P))=0$. Similarly $\mathrm{ord}_3\left(\frac{y(R)}{(1-\omega)x(R)}\right)=\frac{3}{4}$. We also showed in the proof of Corollary \ref{cor2.3.7} that $\mathrm{ord}_3(27-x(P)^3)=4$, so when we add the three terms in equation \eqref{eq3.1}, the product of the denominators has $3$-adic valuation $12$. 
 The numerator is of the form
\begin{align*}\left(27-x(P)^3\right)^2\left(3^6-3^3(y(P)+y(Q)+y(R))\right)+\left(\text{terms of $3$-adic valuation $\geqslant \frac{27}{2}$}\right),
\end{align*}
and $\mathrm{ord}_3(y(P))=\frac{9}{4}$, so
$$\mathrm{ord}_3\left(\sum\limits_{c\in V^{(\chi)}}\mathscr{M}(c,D)\right)\geqslant \left(8+\frac{21}{4}\right)-12=\frac{5}{4}.$$

On the other hand, by the proof of Lemma \ref{lem2.3.11}, we have $9\mid \#(V^{(\chi)})$. Thus,
\begin{align*}
\mathrm{ord}_3\left(\sum\limits_{c\in V^{(\chi)}}\mathcal{E}_1^*\left(\frac{c\O}{D}+\frac{\Omega}{3}, \mathcal{L}\right)\right)&=\min\left(\mathrm{ord}_3\left(\frac{1}{2}\sum\limits_{c\in V^{(\chi)}}\mathscr{M}(c,D)\right), \mathrm{ord}_3(\#(V^{(\chi)}))\right)\\
&\geqslant \frac{5}{4}
\end{align*}
as required.
\end{proof}

\section{Numerical examples}

\begin{minipage}{\textwidth}

The following examples are computed using Magma.

\subsection*{Quadratic Twists}
Let $E(D^3): y^2=4x^3-3^3D^3$, $\omega=e^{\frac{2\pi i}{3}}$. In what follows, $\pi$ denotes a prime of $K$ congruent to $1$ modulo $12$. In particular, $D=\mathrm{N}\pi$ is a special split prime defined in Definition \ref{def2.2.3}. We order $\pi=a+b\omega$, $a,b\in \mathbb{Z}$ ,by $|a|$ and then by $|b|$.

\begin{equation}\nonumber
\begin{array}{lllll}
\pi=a+b\omega& D=\mathrm{N}\pi&L^{(\text{alg})}(E(D^3),1)\\
\hline
 13+    12\omega  & 157 &12=2^2\cdot 3\\
13  +  24\omega &  433 &48=2^4\cdot 3\\
 -23-    12\omega &  397&0\\
-23-    36\omega & 997 &0\\
   25+    24\omega &601 &48=2^4\cdot 3\\
 25+    36\omega &1021 &12=2^2\cdot 3\\
   37+    60\omega&  2749 &12=2^2\cdot 3\\
37 +   72\omega &3889 &0\\
  37+    12\omega&  1069 &12=2^2\cdot 3\\
   47+    12 \omega& 1789 &12=2^2\cdot 3\\
   47+    24 \omega& 1657 &12=2^2\cdot 3\\
   49+    24 \omega& 1801 & 12=2^2\cdot 3\\
   49+    36\omega&  1933 &48=2^4\cdot 3\\
   49+    60\omega &3061& 12=2^2\cdot 3\\

  \hline
\end{array}
\end{equation}
\end{minipage}

\begin{equation}\nonumber
\begin{array}{lll}
\pi=a+b\omega& D=\mathrm{N}\pi&L^{(\text{alg})}(E(D^3),1)\\
\hline
   49+    72\omega & 4057&48=2^4\cdot 3\\   
   -59-    12\omega&  2917 & 0\\
   -59-    48 \omega& 2953 &12=2^2\cdot 3\\
    -59-    60\omega&  3541 &12=2^2\cdot 3\\
   -59-    84\omega & 5581 &48=2^4\cdot 3\\
   61+    24\omega&  2833 &108=2^2\cdot 3^3\\
 61+    72\omega&  4513&108=2^2\cdot 3^3\\
   61+    84\omega & 5653 &12=2^2\cdot 3\\
-71-  132\omega &13093 &12=2^2\cdot 3\\
   73+    96\omega&  7537 &108=2^2\cdot 3^3\\
   73+   108\omega & 9109 &48=2^4\cdot 3\\
   -83-   120\omega &11329 &59=2^4\cdot 3\\
   85+   156\omega &18301&0\\
85 +  168\omega& 21169&192=2^6\cdot 3\\
   -95-   156\omega& 18541 &0\\
   -71-    72 \omega& 5113 &0\\
   -71-    84\omega&  6133 &108=2^2\cdot 3^3\\
   73+    12 \omega& 4597 &0\\
   73+    24 \omega& 4153 &12=2^2\cdot 3\\
   73+    48\omega&  4129 &12=2^2\cdot 3\\
   73+    60 \omega& 4549 &48=2^4\cdot 3\\
   83+    12 \omega& 6037 &12=2^2\cdot 3\\
   -83-    36 \omega& 5197 &48=2^4\cdot 3\\
   -83-    48 \omega& 5209 &12=2^2\cdot 3\\
   85+    48 \omega& 5449 &192=2^6\cdot 3\\
   -95-    24 \omega& 7321 &0\\
   -95-    72 \omega& 7369 &0\\
   -95-    84\omega&  8101 & 0\\
   -95-   108\omega& 10429 &12=2^2\cdot 3\\
   97+    36 \omega& 7213 &12=2^2\cdot 3\\
   97+    48 \omega& 7057 &108=2^2\cdot 3^3\\
   97+    84 \omega& 8317 &12=2^2\cdot 3\\
   97+   108\omega& 10597 &12=2^2\cdot 3\\
   97+   132\omega& 14029 &12=2^2\cdot 3\\
  -107-    60\omega&  8629 &48=2^4\cdot 3\\
  -107-    72\omega&  8929 &48=2^4\cdot 3\\
  107+   120\omega& 13009 &48=2^4\cdot 3\\
  109+    60\omega&  8941 &12=2^2\cdot 3\\
  109+    84\omega&  9781 &48=2^4\cdot 3\\
  109+   144\omega& 16921 &0\\
  109+   156\omega& 19213 &108=2^2\cdot 3^3\\
  -119-    96\omega& 11953 &0\\
  -119-   108\omega& 12973 &48=2^4\cdot 3\\
  -119-   120\omega& 14281 &48=2^4\cdot 3\\
  -119-   132\omega& 15877& 108=2^2\cdot 3^3\\
  -119-   144\omega& 17761 &108=2^2\cdot 3^3\\

\hline

\end{array}
\end{equation}

\begin{equation}\nonumber
\begin{array}{lll}
\pi=a+b\omega& D=\mathrm{N}\pi&L^{(\text{alg})}(E(D^3),1)\\
\hline
  121+    72\omega& 11113 &12=2^2\cdot 3\\
  121+    96\omega& 12241 &0\\
    121+   156\omega& 20101 &48=2^4\cdot 3\\
  121+   180\omega& 25261 &108=2^2\cdot 3^3\\
  -131-   132\omega& 17293 &12=2^2\cdot 3\\
  -131-   156\omega& 21061& 12=2^2\cdot 3\\
 -131-   180\omega& 25981 &108=2^2\cdot 3^3\\
  133+   144\omega& 19273 &48=2^4\cdot 3\\
  133+   156\omega& 21277 &0\\
    -143-   144\omega& 20593 &48=2^4\cdot 3\\
  -143-   180\omega& 27109& 12=2^2\cdot 3\\
  145+   132\omega& 19309 &108=2^2\cdot 3^3\\
  145+   156\omega& 22741 &0\\
  145+   168\omega& 24889 &48=2^4\cdot 3\\
  -155-   144\omega& 22441 &12=2^2\cdot 3\\
  -155-   156\omega& 24181 &108=2^2\cdot 3^3\\
  -155-   168\omega& 26209 &12=2^2\cdot 3\\
  157+   144\omega& 22777 &300=2^2\cdot 3\cdot 5\\
  157+   168\omega& 26497 &0\\
  157+   180\omega& 28789 &12=2^2\cdot 3\\
  -167-   168\omega& 28057 &300=2^2\cdot 3\cdot 5\\
\hline
\end{array}
\end{equation}
\;

The following is a small sample of $D$ divisible by two relatively small (due to computational complexity) distinct special split primes.

\begin{equation}\nonumber
\begin{array}{lll}
 D&L^{(\text{alg})}(E(D^3),1)\\
\hline

157\cdot601&0\\
601\cdot 1021 &0\\
 157\cdot 1021 & 192=2^6\cdot3\\
157\cdot 1789 & 0\\
1021\cdot 1789 & 1200=2^4\cdot 3\cdot 5^2\\
\hline
\end{array}
\end{equation}

\subsection*{Cubic Twists}
Let $E(D^2): y^2=4x^3-3^3D^3$, $\omega=e^{\frac{2\pi i}{2}}$. Let $D$ be an odd, cube-free integer such that $D\equiv 1\bmod 9$ and $D$ is a product of  prime numbers  congruent to $1$ modulo $3$. We first list examples where $D$ is a prime number, $D=\mathrm{N}\pi$ and $\pi$ is a prime of $K$. We order $\pi$ by $|a|$ and then by $|b|$.

\begin{equation}\nonumber
\begin{array}{lll}
\pi=a+b\omega&D=\mathrm{N}\pi&L^{(\text{alg})}(E(D^2),1)\\
\hline
   1+    9\omega & 73  &  9=3^2\\
   1+   18\omega&  307  &  9=3^2\\
 1-27 \omega    &757&   27=3^3 \\
 1+ 81 \omega      &6481&   27=3^3 \\
    4+   15\omega&  181   & 9=3^2\\
   7+   12\omega&  109   & 9=3^2\\
   7+   30\omega &739   &36=2^2\cdot 3^2\\
   7+   39\omega &1297   & 9=3^2\\
         7+   48\omega &2017 &   9=3^2\\
 \hline
\end{array}
\end{equation}

\begin{equation}\nonumber
\begin{array}{lll}
\pi=a+b\omega&D=\mathrm{N}\pi&L^{(\text{alg})}(E(D^2),1)\\
\hline
  13+    6\omega &127 &   0\\
  13+   15\omega&  199  &  9=3^2\\
  13+   24\omega & 433   & 0\\
  16+   39\omega& 1153 &   9=3^2\\
  16+   57\omega& 2593  & 36=2^2\cdot 3^2\\
  19+   27\omega&  577  &  9=3^2\\
  19+   54\omega& 2251  & 36=2^2\cdot 3^2\\
  22+   15\omega&  379  &  0\\
  25+   21\omega&  541  &  9=3^2\\
  25+   39\omega& 1171   & 0\\
  28+    9\omega&  613    &9=3^2\\
  28+   45\omega& 1549 &   9=3^2\\
  31+    6\omega&  811  &  9=3^2\\
  31+   42\omega& 1423 &   9=3^2\\
  34+    3\omega& 1063  &  0\\
  34+   21\omega&  883  &  0\\
  34+   57\omega& 2467 &  36=2^2\cdot 3^2\\
  37+    9\omega& 1117  &  9=3^2\\
  37+   54\omega& 2287  &  9=3^2\\
  40+   51\omega& 2161  &  9=3^2\\
  43+   30\omega& 1459  &  0\\
  43+   39\omega& 1693  &  9=3^2\\
  43+   48\omega& 2089  &  0\\
  43+   57\omega& 2647  &  0\\
    46+    9\omega& 1783  &  9=3^2\\
  49+    6\omega& 2143  &  9=3^2\\
  49+   24\omega& 1801  &  0\\
  49+   33\omega& 1873  & 36=2^2\cdot 3^2\\
  49+   51\omega& 2503  &  9=3^2\\
  49+   60\omega& 3061  &  9=3^2\\
  52+   21\omega& 2053  &  9=3^2\\
 - 53+ 27 \omega    &4969&   27=3^3 \\
- 53 -135 \omega   &13879&   9=3^2\\
  55+   27\omega& 2269  &  0\\
  55+   36\omega& 2341  & 36=2^2\cdot 3^2\\
  55+   54\omega& 2971  & 36=2^2\cdot 3^2\\
  58+   15\omega& 2719  &  9=3^2\\
 58+   33\omega& 2539  &  0\\
 - 80 -27 \omega   &4969&   27=3^3 \\
- 80 -81 \omega    &  6481& 27=3^3 \\
82 -81 \omega     &19927&   243=3^5\\
82+ 135 \omega     & 13879&  9=3^2\\
- 107+54\omega  &20143  &    27=3^3 \\
 - 107+ 135 \omega& 44119&    27=3^3 \\
109 -81 \omega   &27271&   27=3^3 \\
136 -81 \omega    & 36073&  27=3^3 \\
\hline
\end{array}
\end{equation}
\; 

We list some examples where $D$ is divisible by at least two primes which are not necessarily distinct. Again, $D$ is an odd, cube-free integer such that $D\equiv 1\bmod 9$ and $D$ is a product of  prime numbers  congruent to $1$ modulo $3$.
\begin{equation}\nonumber
\begin{array}{llll}
&D&L^{(\text{alg})}(E(D^2),1)\\
\hline
 & 19^2 & 9=3^2 \\
&37^2&9=3^2 \\
 & 163^2 & 9=3^2 \\
 & 631^2 & 9=3^2\\
&7\cdot 211   &27=3^3\\
&7\cdot 2551  &108=2^2\cdot 3^3\\
&7\cdot 1381 &  27=3^3\\
&7\cdot 3037 &  27=3^3\\
& 19\cdot 37 & 27=3^3\\
& 37\cdot 163 & 27=3^3\\
&7\cdot 13\cdot 19 & 0\\
&7\cdot 13\cdot 19\cdot 37&0\\
&7\cdot 13\cdot 31\cdot 61&0\\
& 109\cdot 307 & 27=3^3\\
&19^2\cdot 163 & 27=3^3\\
&19\cdot 163^2 & 27=3^3\\
&19\cdot 37\cdot 163 & 0\\
&19^2\cdot 37\cdot 163 & 81=3^4\\
&19^2\cdot 37^2\cdot 163 & 0\\
&19\cdot 37\cdot 163^2 & 81=3^4\\
&19\cdot 37^2\cdot 163^2 & 0\\
&19^2\cdot 37\cdot 163^2 & 729=3^6\\
&19^2\cdot 37^2\cdot 163^2 & 2916=2^2\cdot 3^4\\
& 7\cdot 139 & 0\\
&79\cdot 139 & 0\\
&7^2\cdot 37^2 & 27=3^3\\
& 19^2\cdot 37^2 & 27=3^3\\
& 37^2\cdot 163^2 & 108=2^2\cdot 3^3\\
&7^2\cdot 13^2\cdot 19^2 & 81=3^4\\
&7^2\cdot 13^2\cdot 19^2\cdot 37^2&972=2^2\cdot 3^5\\
& 127^2 & 9=3^2\\
& 157^2 & 0\\
 & 229^2 & 0\\
 & 307^2 & 36=2^2\cdot 3^2\\
& 397^2 & 144=2^4\cdot 3^2\\
& 691^2 & 0\\
&127^2\cdot 307^2 & 432=2^4\cdot 3^3\\
&127\cdot 307^2 & 54864=2^4\cdot 3^3\cdot 127\\
&127^2\cdot 307 & 8289=3^3\cdot 307\\
&127\cdot 307&0\\
&127^2\cdot 397^2 & 27=3^3\\
& 307^2\cdot 397^2 &2187=3^7\\
\hline
\end{array}
\end{equation}

\end{document}